\documentclass[a4paper]{amsart}

\pdfoutput=1

\usepackage{amsthm, amssymb, amsmath, amsfonts, mathrsfs}

\usepackage{microtype}

\usepackage[pagebackref,colorlinks=true,pdfpagemode=none,urlcolor=blue,
linkcolor=blue,citecolor=blue]{hyperref}

\usepackage{relsize}

\usepackage{amsmath,amsfonts,amssymb,amsthm}
\usepackage{amsthm, amssymb, amsmath, amsfonts, mathrsfs}

\usepackage{mathrsfs}
\usepackage{MnSymbol}
\usepackage{scalerel} 

\usepackage{color}
\usepackage{accents}

\usepackage[normalem]{ulem}
\usepackage{bbm}

\usepackage{cleveref}

\usepackage{mathtools}

\DeclarePairedDelimiter\floor{\lfloor}{\rfloor}

\definecolor{labelkey}{gray}{.8}
\definecolor{refkey}{gray}{.8}

\definecolor{darkred}{rgb}{0.9,0.1,0.1}
\definecolor{darkgreen}{rgb}{0,0.5,0}

\newcommand{\tkcomment}[1]{\marginpar{\raggedright\scriptsize{\textcolor{blue}{#1}}}}

\newtheorem{theorem}{Theorem}[section]
\newtheorem{lemma}[theorem]{Lemma}

\newtheorem{proposition}[theorem]{Proposition}

\theoremstyle{remark}
\newtheorem{remark}[theorem]{Remark}

\renewenvironment{proof}[1][Proof]{ {\itshape \noindent {#1.}} }{$\Box$
\medskip}

\numberwithin{equation}{section}
\newcommand{\R}{\mathbb{R}}
\newcommand{\bbZ}{\mathbb{Z}}
\newcommand{\bbR}{\mathbb{R}}

\newcommand{\Z}{\mathbb{Z}}

\newcommand{\Pb}{\mathbb{P}}
\newcommand{\PP}{\mathbf{P}}
\newcommand{\E}{\mathbb{E}}

\newcommand{\F}{\mathcal{F}}

\newcommand{\B}{\mathcal{B}}

\newcommand{\G}{\mathcal{G}}

\newcommand{\cR}{\mathcal{R}}
\newcommand{\D}{\mathcal{D}}

\newcommand{\eps}{\varepsilon}
\def\les{\lesssim}

\newcommand{\Var}{\mathrm{Var}}

\newcommand{\la}{\langle}
\newcommand{\ra}{\rangle}

\newcommand{\EE}{\mathbf{E}}

\newcommand{\cM}{\mathscr{M}}

\newcommand{\cP}{\mathcal{P}}

\newcommand{\x}{\mathbf{x}}
\newcommand{\bT}{\mathbb{T}}
\newcommand{\cZ}{\mathcal{Z}}

\newcommand{\cY}{\mathcal{Y}}

\newcommand{\rhof}{\rho_{\mathrm{f}}}
\newcommand{\rhob}{\rho_{\mathrm{b}}}
\newcommand{\lf}{\lfloor}
\newcommand{\rf}{\rfloor}

\newcommand{\cal}{\mathcal}
\newcommand{\bx}{\mathbf{x}}
\newcommand{\sfZ}{\mathsf{Z}}
\newcommand{\sfp}{\mathsf{p}}

\begin{document}
\title{Some recent progress on the periodic KPZ equation}

\author{Yu Gu, Tomasz Komorowski}

\address[Yu Gu]{Department of Mathematics, University of Maryland, College Park, MD 20742, USA. }
\email{yugull05@gmail.com}

\address[Tomasz Komorowski]{Institute of Mathematics, Polish Academy of Sciences, ul.
Śniadeckich 8, 00-636 Warsaw, Poland. }
\email{tkomorowski@impan.pl}

\maketitle

\begin{center}{\em In  memory of Giuseppe Da Prato
  \\
   (July 23, 1936 - October 6, 2023)}
\end{center}

\begin{abstract}
We review recent progress on the study of the Kardar-Parisi-Zhang (KPZ) equation in a periodic setting, which describes the random growth of an interface in a cylindrical geometry. The main results include central limit theorems for the height of the interface and the winding number of the directed polymer in a periodic random environment. We present two different approaches for each result, utilizing either a homogenization argument or tools from Malliavin calculus. A surprising finding in the case of a $1+1$ spacetime white noise is that the effective variances for both the height and the winding number can be expressed in terms of independent Brownian bridges.

Additionally, we present two new results: (i) the explicit expression of the corrector used in the homogenization argument, and (ii) the law of the iterated logarithm for the height function.

\medskip


\end{abstract}

\tableofcontents

\section{Introduction}

The Kardar-Parisi-Zhang (KPZ) equation is a stochastic PDE   regarded as the default model of surface growth subject to random perturbations \cite{kardar1986dynamic,HZ95,KS91}. It serves as a prototypical model in nonequilibrium statistical mechanics and is closely related to directed polymers in a random environment, asymmetric exclusion processes, last/first passage percolation, among others. The study of the KPZ equation and related models within its universality class has seen tremendous progress over the past fifteen years, including advances in solution theories and the weak and strong KPZ universality conjectures. For more comprehensive reviews, we refer to \cite{Cor12, Qua12,QS15} and the references cited therein.

One of the most  prominent features in the $1+1$ KPZ universality
class is the 1:2:3 scaling: the fluctuations at time $t\gg1$ are of
order $t^{1/3}$ and the correlation length is of order
$t^{2/3}$. Proving these scaling exponents remains a major challenge
for general models. On the other hand, for models formulated on a
compact domain, it is well-known -- or at least expected -- that due to fast mixing in time, the fluctuations of the height function should be of order $t^{1/2}$ and satisfy the central limit theorem (CLT).   Given these two distinct behaviors depending on the underlying domain, it seems natural to explore the transition between them. We are particularly interested in studying the random fluctuations as both the time and the size of the domain grow together, aiming to understand the transition from $t^{1/3}$ to $t^{1/2}$ 
  for the height fluctuations.

In this paper, we review our studies on the KPZ equation posed on a torus and discuss the proofs of the CLTs for both the height function and the winding number of the related polymer model. These CLTs should be regarded as homogenization-type results, where the limiting variances are viewed as effective diffusivities. One of the main messages we aim to convey is that, to understand the sub-diffusive $t^{1/3}$ 
  or super-diffusive $t^{2/3}$
  behaviors, one can perform a diffusive approximation by posing
  the equation on a large torus and analyze the asymptotic behavior of
  the effective diffusivity. From the perspective of mathematical physics, this can be seen as approximation through an infra-red cutoff. This idea is well-known in the study of
  anomalous behaviors in turbulent transport, where it is also natural
  to examine how the effective diffusivity behaves as the domain under
  consideration expands, see
  e.g. \cite{ABK24,CLF22,CMOW22,HKP-G16,HIKNP-G18,IKNR14,KO02}.

Our studies are closely related to recent works \cite{BL16, BL18,
  BL19, BL21, BLS20}, where the totally asymmetric simple exclusion
process on a large torus was studied in detail, with very precise
fluctuation results derived (see \cite{baik21} for a review). Although
both models involve growth in a cylindrical geometry, the perspectives
we took and the tools we have used are quite different.

One of our main contributions is the derivation of effective diffusivities in terms of the underlying invariant measure, which for the periodic KPZ equation with a $1+1$ spacetime white noise, is the Brownian bridge \cite{BG97, FQ15}. Similar questions can be asked for the open KPZ equation, which has Neumann boundary conditions. There has been significant recent progress in constructing explicit invariant measures \cite{BCY23,  BD21, BKWW21,CK21}, and we believe a central limit theorem similar to Theorem~\ref{t.cltheight} below can be derived for the open KPZ equation. It would be very interesting to study the asymptotics of the effective diffusivity in this case and derive results similar to Theorem~\ref{t.123}, depending on the boundary parameter.

A comprehensive introduction to the subject from a physics perspective can be found in \cite{Pr24}. We will discuss other related mathematical work throughout the paper as we present the main results. In the next section, we define the primary objects studied in this paper.

\subsection{Setup and notations}
To fix the notations, let $Z$ be the solution to the stochastic heat equation (SHE):
\begin{equation}\label{e.she}
\partial_t Z(t,x)=\frac12\Delta Z(t,x)+\beta \xi(t,x) Z(t,x), \quad\quad t>0, x\in\bT^d,
\end{equation}
where $\xi$ is a Gaussian noise over some probability space
$(\Omega,{\mathcal F},\PP)$ that is white in time:
\[
\EE\,\big[ \xi(t,x)\xi(s,y)\big]=\delta(t-s)R(x-y).
\]
Here $\EE$ is the expectation w.r.t. $\PP$ and $\bT^d$ is the $d-$dimensional torus of size
  $L$, which is the interval $[0,L]^d$ with
  identified endpoints. To
simplify the notations, we keep the dependence on $L$ implicit. The parameter $\beta>0$ is a fixed constant and we interpret the product between $\xi$ and $Z$ in \eqref{e.she} in the It\^o-Walsh sense \cite{DK14,walsh}. The function $R$ is the spatial covariance function and throughout the paper we consider two cases: 

(i) $0\leq R\in C^\infty(\bT^d)$ is a smooth function and $\int_{\bT^d} R(x)dx=1$;

(ii) $R(\cdot)=\delta(\cdot)$ in $d=1$ in which case $\xi$ is a $1+1$ spacetime white noise. 

   For the solution theory of \eqref{e.she}, we refer to  the monograph
\cite{daza} for the case (i) and  the lecture notes \cite{DK14,walsh}
for the case (ii). 
Starting from the non-negative initial data $Z(0,x)=Z_0(x)$, one can show that $Z$ stays positive almost surely, and we define 
\begin{equation}\label{e.defh}
h(t,x)=\log Z(t,x),  
\end{equation}
\begin{equation}\label{e.rho}
\rho(t,x)=\frac{Z(t,x)}{\int_{\bT^d}Z(t,x')dx'}.
\end{equation}
The function $h$ is shown to be the physical solution of the KPZ equation \cite{GP17,MH13}:
\begin{equation}\label{e.kpz}
\partial_t h=\frac12\Delta h+\frac12|\nabla h|^2-\frac12\beta^2R(0)+\beta\xi.
\end{equation}
For $R=\delta$ the above equation is only formal since $R(0)=\infty$,
corresponding to the  renormalization  needed to make sense of this
equation. Through a Feynman-Kac formula, which, again, only makes
sense in the case of $R\in C^\infty$, see \cite{BC95}, one can write
the solution $Z$ as
\tkcomment{$\leftarrow$} 
\begin{equation}\label{e.fk}
Z(t,x)=\E_B \big[e^{\beta\int_0^t \xi(t-s,x+B_s)ds-\frac{\beta^2}2R(0)t} Z_0(x+B_t)\big],
\end{equation}
where $B$ is a standard Brownian motion independent of $\xi$ and
$\E_B$ is the expectation on $B$.



 In this way, one can view $e^{\frac{\beta^2}2R(0)t} Z(t,x)$
as the partition function of a directed polymer in the random
environment $\xi$ and $\beta$ as the inverse temperature parameter.
The polymer path $(s,x+B_{t-s})_{0\le s\le t}$ starts at $s=t$
  from $(t,x)$ and runs backward in time, with a ``payoff'' function
  $Z_0(\cdot)$ at $s=0$. With this polymer interpretation, one can view
  $h=\log Z$ as the free energy.


%
%

\subsection{Organization of the paper}

The rest of the paper is organized as follows. In
Section~\ref{s.projective}, we present the results on the
$\{\rho(t,\cdot)\}_{t\geq0}$  process defined in \eqref{e.rho},
through which we also study the average growth speed associated with
\eqref{e.kpz}. In Sections~\ref{s.height} and \ref{s.winding}, we
review the results on the Gaussian fluctuations of the height function
and the winding number of the polymer path. Throughout
Sections~\ref{s.projective} to \ref{s.winding}, we explain the main
ideas of the proofs and provide heuristics, without delving into
technical details, and refer interested readers to the corresponding
papers for specifics.
In Sections~\ref{s.corrector} and \ref{s.lil}, we present two new
results: the explicit expressions for the corrector used in the
homogenization argument and the law of the iterated logarithm for the height function. Finally, we outline a few interesting open problems in Section~\ref{s.open}.

\subsection*{Acknowledgement} We would like to thank Alexander Dunlap
for the collaboration and multiple discussions on the
subject. We also thank the anonymous referee for several helpful suggestions and comments. Y. G. was partially supported by the NSF
through DMS-2203014.  T. K. acknowledges the support of the NCN grant 2020/37/B/ST1/00426.

\section{Polymer endpoint density as the projective process}
\label{s.projective}
Since  $\{Z(t,\cdot)\}_{t\geq0}$ is a Markov process which solves the
linear equation \eqref{e.she}, it is not hard to see
  that the process
$\{\rho(t,\cdot)\}_{t\geq0}$ is also a Markov process, taking values
in the space $D^\infty(\bT^d)$, which is the space containing all continuous probability densities w.r.t. the
  Lebesgue measure  on $\bT^d$. This process is sometimes referred to as the {\em projective process}, as it
can be viewed as the projection of $Z(t,\cdot)$ onto the unit
sphere  in $L^1(\bT^d)$. 
Such processes play a crucial role in the study of $Z$ itself. Roughly speaking, since the random perturbation $\xi$ in \eqref{e.she} was introduced multiplicatively, one might expect that the solution $Z$ either grows to infinity or decays to zero exponentially fast. This exponential growth/decay rate, known as the Lyapunov exponent, is strongly influenced by the long-term behavior of the projective process.
The goal of this section is to discuss various properties of the $\rho$ process. 

We introduce some notations. Let $\G_{t,s}(x,y)$ be the Green's function of the SHE: for any fixed $(s,y)$, it solves
\begin{equation}\label{e.defgreen}
\begin{aligned}
&\partial_t \G_{t,s}(x,y)=\frac12\Delta_x \G_{t,s}(x,y)+\beta \xi(t,x)\G_{t,s}(x,y), \quad\quad t>s, x\in\bT^d,\\
&\G_{s,s}(x,y)=\delta(x-y).
\end{aligned}
\end{equation}
Using $\G$, the solution to SHE can be rewritten as $Z(t,x)=\int_{\bT^d} \G_{t,0}(x,y)Z_0(y)dy$, and for $n\in\Z_+,x_n\in\bT^d$, one can also use the Green's function and write $Z(n,x_n)$ as 
\[
Z(n,x_n)=\int_{\bT^{nd}} 
\G_{n,n-1}(x_n,x_{n-1})\ldots\G_{1,0}(x_1,x_0)Z_0(x_0)dx_0\ldots dx_{n-1}.
\]
It is instructive to view the above expression as a continuous and infinite dimensional analog of the product of $n$ independent random matrices applied to the ``vector'' $Z_0(\cdot)$, generating the new ``vector'' $Z(n,\cdot)$. In other words, starting from the initial function $Z_0$, one iteratively applies the independent random integral operators with kernels given by $\G_{j,j-1}(x_j,x_{j-1})$, and $Z(n,\cdot)$ is the resulting (random) function after $n$ iterations. Similar to the study of the product of i.i.d. random matrices, it is crucial to track the direction of the vector after each iteration, which, in the infinite-dimensional setting, corresponds to the 
$\rho$ process. For a comprehensive introduction to the study of the
product of i.i.d. random matrices, from which one can draw an inspiration and borrow tools for the infinite-dimensional setting, we refer to \cite{BJ12}.

Using $\G$, one can write $\rho$ as 
\[
\rho(t,x)=\frac{Z(t,x)}{\int_{\bT^d} Z(t,x')dx'}=\frac{\int_{\bT^d}\G_{t,0}(x,y)Z_0(y)dy}{\int_{\bT^d}\G_{t,0}(x',y)Z_0(y)dydx'}.
\]
Fix each $t>0$,  $\rho(t,\cdot)$ is the density of the endpoint of a directed
  random polymer of length $t$, distributed initially with density
  $\rho(0,\cdot)=Z_0(\cdot)/\|Z_0\|_{L^1(\bT^d)}$, with $\xi$ being the random potential. To study $\rho$, we may as well assume that $\int_{\bT^d} Z_0(y)dy=1$ -- otherwise we normalize it which does not change the values of $\rho$. For the rest of the discussion, it is convenient to introduce the following notations on the endpoint densities
of the ``forward'' and ``backward'' polymer path on a cylinder. Let ${\mathcal M}_1(\bT^d)$ be the set of all Borel probability
measures  on $\bT^d$, considered with the topology of weak convergence
  of measures.
For
any $\nu\in{\mathcal M}_1(\bT^d)$ and $t>s$, we define 
\begin{equation}\label{e.forwardbackward}
\begin{aligned}
&\rhof(t,x;s,\nu)=\frac{\int_{\bT^d} \G_{t,s}(x,y)\nu(dy)}{\int_{\bT^{2d}}\G_{t,s}(x',y)\nu(dy)dx'}, \quad x\in\bT^d,\\
&\rhob(t,\nu;s,y)=\frac{\int_{\bT^d}\G_{t,s}(x,y)\nu(dx)}{\int_{\bT^{2d}}\G_{t,s}(x,y')\nu(dx)dy'},\quad y\in\bT^d.
\end{aligned}
\end{equation}
Here the subscripts ``$\mathrm{f}, \mathrm{b}$'' represent ``forward'' and ``backward'' respectively. When $\nu$ is a Dirac measure on $\bT^d$, we   write
$$\rhof(t,x;s,y)=\rhof(t,x;s,\delta_y)\quad \mbox{and}
\quad \rhob(t,x;s,y)=\rhob(t,\delta_x;s,y).
$$
Using the above notations, we can write $\rho(t,x)$ as 
\[
\rho(t,x)=\rhof(t,x;0,\nu), \quad\quad \mbox{ with } \quad\quad \nu(dx)=\tfrac{Z_0(x)dx}{\int_{\bT^d}Z_0(x')dx'}.
\]
For fixed $t>s$, we also have (\cite[Eq. (B.1)]{ADYGTK22})
\[
\{\G_{t,s}(x,y)\}_{x,y\in\bT^d}\stackrel{\text{law}}{=}\{\G_{t,s}(y,x)\}_{x,y\in\bT^d},
\]
from which we conclude that 
\[
\{\rhof(t,x;s,\nu)\}_{x\in\bT^d}\stackrel{\text{law}}{=}\{\rhob(t,\nu;s,x)\}_{x\in\bT^d}.
\]

\textbf{A semi-martingale decomposition.} To see how the $\rho$ process arises in the study of $\log Z$, we consider the following stochastic calculus trick that has been used in the study of spin glass models  \cite{CN95}. Define 
\begin{equation}\label{e.defbarZ}
\bar{Z}_t=\int_{\bT^d} Z(t,x)dx,
\end{equation}
which can be viewed as the partition function of the point-to-line
directed polymer. We first note that the projective process $\rho$ can
be written as $\rho=Z/\bar{Z}$.
Suppose that $q_t(x)$ is the heat kernel on $\bT^d$, i.e.
  \begin{equation}\label{e.defptx}
q_t(x)=\frac{1}{(2\pi t)^{d/2}}\sum_{n\in\bbZ^d}\exp\left\{-\frac{|x+n|^2}{2t}\right\}.
\end{equation} 
From the mild formulation of the SHE,
\[
Z(t,x)=q_t\star Z_0(x)+\beta \int_0^t\int_{\bT^d}q_{t-s}(x-y)Z(s,y)\xi(s,y)dyds,
\]
we can show that $\bar{Z}_t$ is a continuous positive martingale, given
by:
\begin{equation}\label{e.barZeq}
\bar{Z}_t=\int_{\bT^d} Z(t,x)dx=\int_{\bT^d} Z_0(x)dx+\beta\int_0^t \int_{\bT^d} Z(s,y)\xi(s,y)dyds.
\end{equation}
Now define $M$ through $M_0=0$ and the stochastic differential 
\begin{equation}\label{e.defM}
dM_t=\bar{Z}_t^{-1} d\bar{Z}_t.
\end{equation}
 One can show that $M$ is a martingale and, using the It\^o formula, it
can be concluded that
\begin{equation}
  \label{010608-24}
\log \bar{Z}_t=\log \bar{Z}_0+ M_t-\tfrac12\la M\ra_t.
\end{equation}
Thus,
the following representation of $\bar{Z}_t$ holds
\[
\bar{Z}_t=\bar{Z}_0 \exp(M_t-\tfrac12\la M\ra_t).
\]
From \eqref{010608-24} we also have a natural semi-martingale
decomposition of $\log \bar{Z}_t$, in terms of the projective process
$\rho=Z/\bar{Z}$, with
\begin{equation}\label{e.semilogZ}
\begin{aligned}
&M_t=\beta \int_0^t\int_{\bT^d}\rho(s,y)\xi(s,y)dyds,\\
& \la M\ra_t=\beta^2\int_0^t\int_{\bT^{2d}} \rho(s,y_1)\rho(s,y_2)R(y_1-y_2)dy_1 dy_1ds=:\beta^2\int_0^t \mathcal{R}(\rho(s))ds.
\end{aligned}
\end{equation}
Here the quadratic functional $\mathcal{R}$ is defined as 
\begin{equation}\label{e.defoverlap}
\mathcal{R}(f)=\int_{\bT^{2d}}f(x_1)f(x_2)R(x_1-x_2)dx_1dx_2.
\end{equation}
The functional is defined for 
$f\in L^1(\bT^d)$ in case $R(\cdot)$ is smooth and for $f\in
L^2(\bT)$, if $R=\delta$. Since $R(\cdot)$ is a covariance function, the functional is
non-negative. It can be viewed as a weighted  
semi-norm that measures the overlap of the two ``replicas''
of random polymers. To study the long time behavior of $\log
\bar{Z}_t$, one can start from its representation in \eqref{010608-24}
together with \eqref{e.semilogZ}, where the martingale $M$ is the
It\^o integral of $\rho$ with respect to $\xi$ and the drift term
$-\tfrac12\la M\ra_t$ is an additive functional of   $\rho$. To study
the long time behavior of the height function $\log Z(t,x)$, we can write it as 
\begin{equation}\label{e.logZtx}
\log Z(t,x)=\log \bar{Z}_t+\log \rho(t,x).
\end{equation}
In other words, we have expressed all the quantities of interest
in terms of the projective process $\rho$.

To study the {\em Lyapunov exponent}, i.e., the limit of
$\tfrac{1}{t}\log \bar{Z}_t$ as $t\gg1$, a standard concentration
argument (see, e.g., \cite[Section 2]{RT05}) shows that it suffices to
consider its expectation. Consequently, the main contribution comes
from the term $-\tfrac{1}{2t}\EE\la M\ra_t$. If the $\rho$ process
is asymptotically stable, i.e., $\rho(t)$ converges to some invariant
distribution as $t\gg1$, one may hope to express the Lyapunov exponent in terms of this invariant distribution. Therefore, studying the quantitative properties of the Lyapunov exponent reduces to analyzing the invariant measure of the projective process. For more details, see Section~\ref{s.lyapunov} below.

To study the random fluctuations of $\log \bar{Z}_t$, the first step is to demonstate  that the Markov process $\rho$ mixes sufficiently fast. Heuristically, the spectral gap of $\rho$ should be of order $L^{-3/2}$, indicating that, for fixed $L$, the process mixes exponentially fast. With this rapid mixing, it is natural to expect that the  random vector $(M_t, \la M\ra_t)$, after a centering and rescaling, converges in distribution to a two-dimensional (correlated) Gaussian. For the martingale $M_t$, one can apply the martingale central limit theorem; thus, the main challenge is analyzing the additive functional  $\la M\ra_t$. In Section~\ref{s.homogenization}, we will present a  homogenization type argument, through solving a corrector equation and performing a martingale decomposition. 

To summarize, the projective process $\rho$ naturally arises  in the study of $\log Z$. In the periodic setting, it can be interpreted as the endpoint density of the directed polymer on a cylinder, and its rapid relaxation to equilibrium is expected. Throughout our studies of the periodic KPZ, the process $\rho$ plays a central role. In nearly all our calculations, we can approximate it by its stationary version due to this rapid relaxation, a topic we discuss in detail in the next section.

\subsection{Invariant measure and exponential mixing}
\label{s.sinai}
We   prove the exponential mixing of the Markov process $\rho$, presented in the form of the following one-force-one-solution principle:
\begin{theorem}\label{t.ofos}
The Markov process $\{\rhof(t,\cdot;0,\nu)\}_{t\geq0}$, taking values
in ${\cal M}_1(\bT^d)$, has a unique invariant  probability measure, denoted by $\pi_\infty$. In the case of a $1+1$ space-time white noise, the invariant measure is  given by
the law of the ${\cal M}_1(\bT)$-valued random variable $\varrho(y)dy$,
where
\begin{equation}
   \label{rho-i}
 \varrho(y)=\frac{e^{\beta W(y)}}{\int_{\bT}e^{\beta W(y')}dy'},\quad y\in\bT,
 \end{equation}
 and $W$ is a standard Brownian bridge connecting
 $(0,0)$ and $(L,0)$, i.e.  a continuous
   trajectory, zero mean Gaussian
   process with the covariance function
   $$
\EE[W(x)W(y)]=\min(x,y)-\frac{xy}{L},\quad x,y\in[0,L].
 $$

 Furthermore, for any $p\in[1,\infty)$, there exist $C,\lambda>0$, depending on $R(\cdot),d,\beta,L$, such that for all $t\geq1$, 
\begin{equation}
  \label{051705-23}
\EE\Big[ \sup_{\nu,\nu'\in\mathcal{M}_1(\bT^d)} \sup_{x\in\bT^d} |\rhof(t,x;0,\nu)-\rhof(t,x;0,\nu')|^p  \Big]\leq Ce^{-\lambda t}, 
\end{equation}
and
\begin{equation}\label{e.mmbdrho}
\EE \Big[\sup_{\nu\in \mathcal{M}_1(\bT^d)}\sup_{x\in\bT^d}\, \{\rhof(t,x;0,\nu)^p
+\rhof(t,x;0,\nu)^{-p}\}\Big] \leq C.
\end{equation} 
\end{theorem}

Before discussing the proof of the above result, we make a few   remarks:

(i) In the compact setting, the study of the invariant measure of the
projective process $\rho$ is equivalent to studying   an analogous
problem for the KPZ equation. Consider the solution $Z$ to the SHE \eqref{e.she} with measure-valued initial data $\nu$, and the KPZ solution $h=\log Z$. We have the following  identities 
\begin{align}
  &h(t,x)-h(t,0)=\log \rhof(t,x;0,\nu)-\log \rhof(t,0;0,\nu),
  \label{020608-24a}\\
&\rhof(t,x;0,\nu)=e^{h(t,x)-h(t,0)}/\int_{\bT^d}
e^{h(t,x')-h(t,0)}dx'.
\label{020608-24b}
\end{align}
In other words,   $\rhof(t,\cdot;0,\nu)$ and $h(t,\cdot)-h(t,0)$ can
be expressed in terms of each other. Thus, Theorem~\ref{t.ofos}
implies the uniqueness of the invariant measure for the periodic KPZ
equation, as discussed in \cite{HM18,Ros21}. For the open KPZ equation
with the Neumann   boundary condition, a similar argument yields the
same result \cite{Pa22}, see also \cite{KM22}.  The compactness of the
domain is crucial, since the analogue of the process appearing on the
right hand side of  \eqref{020608-24b}
may degenerate for the equation on the whole
real line. Then, the law of the two-sided Brownian motion serves as an
invariant measure, but in this instance, $\int_{\bbR} e^{h(t,x')-h(t,0)}dx'$ equals infinity.

(ii) The estimate \eqref{051705-23} can be viewed as a quantitative
one-force-one-solution principle, sometimes referred to as
synchronization in the study of random dynamical systems. It  implies
both the existence and uniqueness of the invariant measure. More
specifically,  consider $\rhof$ at time $t$, with some initial value
at time $s<t$. As we pull back the initial time $s\to-\infty$, the random element
$\rhof(t,\cdot;s,\nu)$ converges strongly to a limit $\rhof(t,\cdot)$, which does not
depend on the initial data $\nu$ and is a measurable w.r.t.  $\xi(s',\cdot)$, $s'\le t$. To see it through a Cauchy sequence argument, we just note that there exist $C,\lambda>0$, as in Theorem
\ref{t.ofos}, such that:
\[
\EE\Big[ \sup_{\nu,\nu'\in\mathcal{M}_1(\bT^d)} \sup_{x\in\bT^d} |\rhof(t,x;s,\nu)-\rhof(t,x;s',\nu')|^p \Big] \leq Ce^{-\lambda (t-s)}, \quad\quad t>s\geq s'.
\]
Obviously, the law of the limit $\rhof(t,\cdot)$   is given by $\pi_\infty$. 

(iii) The way we use Theorem~\ref{t.ofos} is  as
follows.  Suppose that
$\nu\in{\cal M}_1(\bT^d)$ is fixed. In \eqref{051705-23}, if we choose $\nu'$ as a random probability measure sampled from the invariant measure $\pi_\infty$ and independent of $\xi$, then $\rhof(t,\cdot;0,\nu')$ has the same law as $\pi_\infty$. The key estimate \eqref{051705-23} allows us to approximate the endpoint density $\rhof(t,\cdot;0,\nu)$  by its stationary version $\rhof(t,\cdot;0,\nu')$, with an error that is exponentially small in $t$.

\medskip

\textbf{Sketch of the proof of Theorem~\ref{t.ofos}.} The proof was
inspired by Sinai's  work \cite{Sin91} on the Burgers equation driven
by a smooth random periodic forcing. Sinai studied the object
$u=\nabla h=\nabla Z/Z$.
In comparison, in our case it is the
projective process
$\rho=Z/\|Z\|_{L^1(\bT^d)}$.
For  both of these processes it is
possible
to remove    a large random constant that is common for   both the
numerator and  
denominator  and represent each process as a ratio of  processes that
can be approximated by random elements    independent of the initial
data (measurable only w.r.t. the noise). This leads to
stabilization in large time as described in Theorem \ref{t.ofos}. Using $\rho$ rather than $u$
 avoids dealing with $\nabla Z$, which is not function-valued in the singular case. For different approaches we refer to \cite{GP20, Ros21}.
 
Sinai's approach involves first a construction of a Markov chain as follows. Fix $t>1$ and for simplicity assume $t=N\in \Z_+$, we have
\begin{equation}\label{e.3111}
Z(t,x)=\int_{\bT^d} \G_{t,t-1}(x,y_1) Z(t-1,y_1) dy_1.
\end{equation}
Iterating the above formula, we derive 
\begin{equation}\label{e.31121}
\begin{aligned}
Z(t,x)=  \int_{\bT^{Nd}} \prod_{j=0}^{N-1} \mathcal{G}_{t-j, {t-j-1}}(y_j ,y_{j+1} ) Z_0(y_N) dy_1\ldots dy_{N-1}dy_N,
\end{aligned}
\end{equation}
where $y_0=x$ and $Z_0$ is the initial data for \eqref{e.she}.   The Markov chain $\{Y_n\}_{n=1}^N$, which takes values in $\bT^d$, is constructed as follows. We run the chain backward in time. Let $\pi_N(y_N)$ be the density of $Y_N$ and $\pi_k(y_k\,|\,y_{k+1})$ be the transition density from $Y_{k+1}=y_{k+1}$ to $Y_k=y_k$:
\begin{equation}\label{e.trandensity}
\begin{aligned}
&\pi_1(y_1\,|\,y_2)=\frac{\G_{t-1,t-2}(y_1,y_2)}{\int_{\bT^d} \G_{t-1,t-2}(y_1,y_2)dy_1},\\
&\pi_k(y_k\,|\,y_{k+1})=\frac{\int_{\bT^{(k-1)d}} \G_{t-1,t-2}(y_1,y_2)\ldots \G_{t-k,t-k-1}(y_k,y_{k+1})dy_1\ldots dy_{k-1}}{\int_{\bT^{kd}} \G_{t-1,t-2}(y_1,y_2)\ldots \G_{t-k,t-k-1}(y_k,y_{k+1})dy_1\ldots dy_{k-1}dy_k},\quad\quad k\leq N-1.
\end{aligned}
\end{equation}
and 
\[
\pi_N(y_N)=\frac{\int_{\bT^{(N-1)d}}\left(\prod_{j=1}^{N-1} \mathcal{G}_{t-j,t-j-1}(y_j ,y_{j+1} ) \right)Z_0(y_N )dy_1\ldots dy_{N-1}}{\int_{\bT^{Nd}}\left(\prod_{j=1}^{N-1} \mathcal{G}_{t-j,t-j-1}(y_j ,y_{j+1} ) \right)Z_0(y_N )dy_1\ldots dy_{N}}.
\]
We first note that the joint density of $Y_1,\ldots,Y_N$ is 
\[
\pi_1(y_1\,|\,y_2)\ldots \pi_{N-1}(y_{N-1}\,|\,y_N)\pi_N(y_N)=A\left(\prod_{j=1}^{N-1} \mathcal{G}_{t-j,t-j-1}(y_j ,y_{j+1} ) \right)Z_0(y_N ),
\] 
with the (random)  normalizing constant $A$ given by 
\[
A=\left(\int_{\bT^{Nd}}\left(\prod_{j=1}^{N-1} \mathcal{G}_{t-j,t-j-1}(y_j ,y_{j+1} ) \right)Z_0(y_N )dy_1\ldots dy_{N}\right)^{-1}.
\]
Let $\E_\pi$ denote the expectation of the Markov chain, then \eqref{e.31121} can be written as
\begin{equation}\label{e.UY1}
\begin{aligned}
Z(t,x)=&\int_{\bT^{Nd}}\G_{t,t-1}(x,y_1)\pi_1(y_1\,|\,y_2)\ldots \pi_{N-1}(y_{N-1}\,|\,y_N)\pi_N(y_N)A^{-1}dy_1\ldots dy_N\\
=&A^{-1}\E_\pi[\G_{t,t-1}(x,Y_1)],
\end{aligned}
\end{equation}
thus,
\begin{equation}\label{e.densityY1}
\rho(t,x)=\frac{Z(t,x)}{\int_{\bT^d}Z(t,x')dx'}=\frac{\E_\pi[\G_{t,t-1}(x,Y_1)]}{\int_{\bT^d}\E_\pi[\G_{t,t-1}(x',Y_1)]dx'}.
\end{equation}
As already mentioned, the key here is that the random constant $A$ disappears when we take the ratio. The same applies to the Burgers equation:
\[
u(t,x)=\frac{\nabla Z(t,x)}{Z(t,x)}=\frac{\E_\pi[\nabla_x\G_{t,t-1}(x,Y_1)]}{\E_\pi[\G_{t,t-1}(x,Y_1)]}.
\] 
The dependence of $\rho(t,x)$ on the random environment $\{\xi(s,\cdot):s\leq t-1\}$ is only through the random variable $Y_1$, therefore, to study the mixing property of $\rho$,  it suffices to study the mixing property of the Markov chain. 

We note that 
\[
\begin{aligned}
&\inf_{y_k,y_{k+1}\in\bT^d}\pi_k(y_k\,|\,y_{k+1})\\
&\geq  \frac{\inf_{y_{k-1},y_k,y_{k+1}\in\bT^d}\G_{t-k+1,t-k}(y_{k-1},y_k)\G_{t-k,t-k-1}(y_k,y_{k-1}) }{\sup_{y_{k-1},y_{k+1}\in\bT^d}\int_{\bT^d}\G_{t-k+1,t-k}(y_{k-1},y_k)\G_{t-k,t-k-1}(y_k,y_{k+1})dy_k  },
\end{aligned}
\] 
and the r.h.s. only depends on the noise in the time interval $[t-k-1,t-k+1]$. For any $s$, the random variable $\G_{s,s-1}(x,y)$ is roughly the periodic heat kernel $q_1(x-y)$ multiplying a random factor that can be statistically bounded away from zero and infinity. Since $q_1(x-y)$ is also bounded away from zero and infinity, which relies on the fact that the spatial domain is compact, it is not hard to imagine that, with a positive probability, the transition density $\pi_k$ can be bounded from below by a positive constant. More precisely, one can show that, there exist a $\delta>0$ depending on $d, R(\cdot),L$ and 
  events $B_{t-k+1,t-k-1}(\delta)$ only involving the random environment $\{\xi(s,\cdot): s\in[t-k-1,t-k+1]\}$  such that 
\begin{equation}\label{e.doeblin1}
\PP[B_{t-k+1,t-k-1}(\delta)]>\delta,
\end{equation}
 and on the event $B_{t-k+1,t-k-1}(\delta)$ the transition
probability kernel satisfies
\begin{equation}\label{e.doeblin2}
\inf_{y_k,y_{k+1}\in\bT^d}\pi_k(y_k\,|\,y_{k+1})>\delta.
\end{equation}
For precise statements, we refer to \cite[Section 4.2 and 4.3]{GK21}.

Since the Markov chain has a random transition kernel, one can view
\eqref{e.doeblin1} and \eqref{e.doeblin2} as a randomized Doeblin
condition. Then it is rather standard to  construct a coupling to show
that the Markov chain mixes  exponentially fast. For any $t\gg1$, we consider the subintervals of of the form $[t-k-1,t-k+1]$ in $[0,t]$, and if one of those events $B_{t-k+1,t-k-1}(\delta)$ occurs, then the transition kernel $\pi_k$ is bounded from below by $\delta$, so we can write 
\[
\pi_k(y_k\,|\,y_{k+1})=\delta+(1-\delta)\tfrac{\pi_k(y_k\,|\,y_{k+1})-\delta}{1-\delta}.
\]
Let us introduce a sequence of  i.i.d. Bernoulli random variables $(\tau_k)$, with
parameter $\delta$, that is independent of the noise $\xi$. Given
$Y_{k+1}=y_{k+1}$, to sample from $\pi_k(y_k\,|\,y_{k+1})$, it is
equivalent with sampling from the uniform distribution if $\tau_k=1$
and from $\tfrac{\pi_k(y_k\,|\,y_{k+1})-\delta}{1-\delta}$ if
$\tau_k=0$. Thus, with probability $\delta$, the distribution of $Y_k$
given $Y_{k+1}=y_{k+1}$ is uniform, which renews the chain. We can also   see that the probability of not renewing is exponentially small in $t$. This leads to the fast mixing of the $\rho$ process. Modulo technical details, it is almost a proof of \eqref{051705-23}. For a complete proof, we refer to \cite[Section 4.4 and 4.5]{GK21} and the proof of \cite[Proposition 4.7]{GK21}.

\begin{remark}
  For a similar equation with no comparison principle (with $-\Delta$ replaced by $(-\Delta)^\alpha$ for appropriate $\alpha$),  results similar to Theorem~\ref{t.ofos} were derived in \cite{HR23}, using a different approach.
\end{remark}

\subsection{A nonlocal nonlinear SPDE, a PDE hierarchy and the Lyapunov exponent}

In this section, we summarize several other known results on the
projective process $\rho$, including: (i) a nonlocal, nonlinear SPDE
that it satisfies, at least in the case of smooth noise; (ii) a PDE
hierarchy satisfied by its multi-point covariance function; and (iii)
the study of the Lyapunov exponent through the invariant measure of
the process.

 \subsubsection{A reaction-diffusion SPDE} Since $\rho$ is a Markov
 process derived from the solution to the SPDE \eqref{e.she}, it is
 natural to inquire if $\rho$ itself solves some SPDE. Recall that
 $\rho=Z/\bar{Z}$, where $Z$ solves the SHE and
 $\bar{Z}_t=\int_{\bT^d} Z(t,x)dx$ is a positive martingale satisfying
 \eqref{e.barZeq}. One can formally apply the It\^o formula, which
 (remembering that $\int_{\bT^d}\rho(t,x')dx'=1$) yields  the following  SPDE:
\begin{equation}\label{e.spderho}
\begin{aligned}
\partial_t \rho(t,x)=&\tfrac12\Delta \rho(t,x) +\beta^2\rho(t,x)\int_{\bT^d} \rho(t,x')\Big[ R\star \rho(t,x')-R\star \rho(t,x)\Big]dx'\\
&+\beta \rho(t,x)\int_{\bT^d}\rho(t,x')[\xi(t,x)-\xi(t,x')]dx', \quad\quad t>0, x\in\bT^d.
\end{aligned}
\end{equation}
With smooth initial data and $R\in C^\infty(\bT^d)$, it was shown in
\cite[Proposition 3.2]{GK21} that $\rho=Z/\bar{Z}$ is indeed a strong
solution to \eqref{e.spderho}. For $R(\cdot)=\delta(\cdot)$, we do not
expect $\rho$ to be a strong solution -- after all, $Z$ solves
\eqref{e.she} only in the mild sense -- however, we expect that, with
some extra work, one can show that $\rho$ is the unique mild solution to \eqref{e.spderho}.

The SPDE \eqref{e.spderho} is of reaction-diffusion type and includes
a nonlocal reaction term. It can be shown that, at least formally,
this SPDE preserves a density: if $\rho(0,\cdot)$ is a probability
density on $\bT^d$, then for any $t>0$,   $\int_{\bT^d}
\rho(t,x)dx=1$.  The nonlinear part of the drift, which in the case
of space-time white noise equals
$$
\beta^2[\rho(t,x)\|\rho(t,\cdot)\|_{L^2(\bT)}^2-\rho(t,x)^2],
$$
measures how localized the density is. In other words, the evolution
of the endpoint density is driven by the localization of the density
itself, which we find intriguing. The equation \eqref{e.spderho} also
applies to $\R^d$, raising the question of whether one can study
properties of the endpoint density, such as the speed of spreading and
localization behaviors, directly through the equation rather than via
the formula  $\rho=Z/\bar Z$ for the projective process.

\begin{remark}
If we consider the equation \eqref{e.spderho} on the whole space and only keep   the drift part, it becomes deterministic and takes the form 
\begin{equation}\label{e.detereqrho}
\partial_t \rho(t,x)=\tfrac12\Delta \rho(t,x) +\beta^2\rho(t,x)\int_{\bT^d} \rho(t,x')\Big[ R\star \rho(t,x')-R\star \rho(t,x)\Big]dx'.
\end{equation}
Starting from a probability density, $\rho(t,\cdot)$ remains a
probability density for all $t>0$. In $d=1$ and for
$R(\cdot)=\delta(\cdot)$ or $R\in C_c^\infty(\R)$, it was shown in
\cite{GH21,GH22} that the mass spreads super-diffusively with the
exponent $2/3$.  More precisely,  with compactly supported initial data, for any $p\geq1$:
\[
\int_{\R} |x|^p \rho(t,x)dx\asymp t^{2p/3}, \quad\quad\mbox{ for } t\gg1.
\]
The symbol $\asymp$ means that the ratio of both sides stays bounded
away from $0$ and $\infty$ for $t\gg1$.
It is unclear to us if there is any relation between
\eqref{e.detereqrho} and \eqref{e.spderho},   or if the obtained exponent $2/3$ above is merely a coincidence.
\end{remark}

\subsubsection{A PDE hierarchy} From Theorem~\ref{t.ofos} and the
discussion in Section~\ref{s.sinai}, we know that the Markov process
$\rho$ has a unique invariant measure $\pi_\infty$, which is closely
related to the invariant measure of the KPZ equation. In the white
noise case where $R(\cdot)=\delta(\cdot)$, the invariant measure is
explicit, expressed in terms of a Brownian bridge  in \eqref{rho-i},
 see \cite[Theorems 1.1 and 2.1]{FQ15}. Unfortunately, this is probably the only case in which we have any concrete knowledge about the invariant measure. As   will become clear later in Section~\ref{s.clarkocone} below, having a good understanding of $\pi_\infty$ is crucial for studying the universal behaviors of growth models in a cylinder geometry. This remains a major challenge in the study of the KPZ universality class: for non-solvable models, there is essentially no knowledge of the underlying invariant measure. 

Suppose we take expectation on both sides of \eqref{e.spderho}, the martingale terms disappear,  and, because of the nonlinearity,  the evolution of $\EE \rho(t,x)$ is related to the two and three point covariance functions. This inspires us to define the $n-$point covariance function of $\rho$: 
\begin{equation}\label{e.defQn}
Q_n(t,\mathbf{x}_{1:n})=\EE \Big[\rho(t,x_1)\ldots \rho(t,x_n)\Big], \quad\quad \mbox{ with } \mathbf{x}_{1:n}=(x_1,\ldots,x_n).
\end{equation}
It was shown in \cite[Theorem 1.1]{GH21} that $\{Q_n\}_{n\geq1}$ solves a PDE hierarchy of the form \footnote{The result in \cite{GH21} was for the equation posed on the whole space; the same proof applies to $\bT^d$.}
\begin{equation}\label{e.hierarchy}
\begin{split}
\partial_tQ_n&(t,\mathbf{x}_{1:n})=\tfrac12\Delta Q_n(t,\x_{1:n})+\beta^2Q_n(t,\x_{1:n})\textstyle\sum_{1\leq i<j\leq n}R(x_i-x_j)\\
&-\beta^2n\textstyle\int_{\bT^d}Q_{n+1}(t,\x_{1:n},x_{n+1})\sum_{i=1}^nR(x_i-x_{n+1})dx_{n+1}\\
&+\beta^2\tfrac{n(n+1)}{2}\textstyle\int_{\bT^{2d}} Q_{n+2}(t,\x_{1:n},x_{n+1},x_{n+2})R(x_{n+1}-x_{n+2})dx_{n+1}dx_{n+2}.
\end{split}
\end{equation}
Here
$\Delta =\sum_{j=1}^n\Delta_{x_j}$.
For $n=1$, one can check the above equation by directly taking expectation on both sides of \eqref{e.spderho}. For a slightly different model, a similar hierarchy was derived in \cite[Theorem 3.1]{CH06}, which was used to study localization behaviors of the polymer measure.

In the compact setting, one can utilize \eqref{e.hierarchy} to study properties of $\pi_\infty$. Since $\rho(t,\cdot)$ converges in law to $\pi_\infty$ as $t\to\infty$, it is natural to expect that $Q_n(t,\cdot)$ defined in \eqref{e.defQn} also converges as $t\to\infty$, with the limit being the $n-$point covariance function of $\pi_\infty$, which solves the stationary version of \eqref{e.hierarchy}. Indeed, we have the following result  \cite[Proposition 2.2]{GK211}
\begin{proposition}
Assume  $\varrho$ is sampled from $\pi_\infty$, and define 
\begin{equation}\label{e.defcalQ}
\mathcal{Q}_n(\bx_{1:n})=\mathcal{Q}_n(x_1,\ldots,x_n)=\EE\Big[\varrho(x_1)\ldots\varrho(x_n)\Big],
\end{equation}
then $\{\mathcal{Q}_n\}_{n\geq1}$ is a weak solution to
\begin{equation}\label{e.hierarchystationary}
\begin{split}
\tfrac12\Delta \mathcal{Q}_n&(\x_{1:n})+\beta^2\mathcal{Q}_n(\x_{1:n})\textstyle\sum_{1\leq i<j\leq n}R(x_i-x_j)\\
&=\beta^2n\textstyle\int_{\bT^d}\mathcal{Q}_{n+1}(\x_{1:n},x_{n+1})\sum_{i=1}^nR(x_i-x_{n+1})dx_{n+1}\\
&-\beta^2\tfrac{n(n+1)}{2}\textstyle\int_{\bT^{2d}} \mathcal{Q}_{n+2}(\x_{1:n},x_{n+1},x_{n+2})R(x_{n+1}-x_{n+2})dx_{n+1}dx_{n+2}.
\end{split}
\end{equation}
For $R\in C^\infty(\bT^d)$, $\{\mathcal{Q}_n\}_{n\geq1}$ are smooth functions and a classical solution to \eqref{e.hierarchystationary}.
\end{proposition}

The above hierarchy provides some partial
 information about the invariant measure
$\pi_\infty$. For example, one can perform an expansion of
$\mathcal{Q}_n$  in  small $\beta\ll1$, for any $n$, up to an arbitrary high order, see \cite[Section 2-3]{GK211}. This can be used to study the expansion of the Lyapunov exponent for $\beta\ll1$, see the discussion in the next section. 

If we instead consider the multi-point covariance function of SHE, 
\[
\mathscr{Q}_n(t,\bx_n):=\EE\Big[ Z(t,x_1)\ldots Z(t,x_n)\Big],
\] it solves the $n-$particle delta Bose gas:
\[
\partial_t\mathscr{Q}_n(t,\mathbf{x}_{1:n})=\tfrac12\Delta \mathscr{Q}_n(t,\x_{1:n})+\beta^2\mathscr{Q}_n(t,\x_{1:n})\textstyle\sum_{1\leq i<j\leq n}R(x_i-x_j).
\]
This is usually the starting point of the replica method,  based on studying the principle eigenvalue of the operator $\frac12\Delta+\sum_{i<j}R(x_i-x_j)$ and performing an analytic continuation to investigate the cumulants of $\log Z$. A related variational approach was proposed by Mezard-Parisi to study directed polymers, and it appears to be closely connected to the above PDE hierarchy (see \cite[Equation (2.20) and (3.2)]{MP92}).

\subsubsection{Lyapunov exponent}
\label{s.lyapunov}

Recall that one of our goals is to study the asymptotics of $\log Z(t,x)$ for large $t$. By \eqref{e.logZtx} and  the semi-martingale decomposition of $\log \bar{Z}_t$ in \eqref{e.semilogZ}, we have 
\[
\begin{aligned}
\EE \log Z(t,x)&= \EE \log \bar{Z}_t+\EE \log \rho(t,x)\\
&=-\tfrac12\beta^2\EE\int_0^t \mathcal{R}(\rho(s))ds +\EE \log \rho(t,x).
\end{aligned}
\]
Applying Theorem~\ref{t.ofos}, we conclude that 
\begin{equation}\label{e.congamma}
\begin{aligned}
\tfrac{1}{t} \EE \log Z(t,x)\to \gamma_L(\beta):&=-\tfrac12\beta^2 \int_{\bT^{2d}} R(x_1-x_2) \mathcal{Q}_2(x_1,x_2)dx_1dx_2\\
&=-\tfrac12\beta^2 \int_{\bT^{2d}} R(x_1-x_2) \EE
\Big[\varrho(x_1)\varrho(x_2)\Big]dx_1dx_2\\
&=-\tfrac12\beta^2 L^d\int_{\bT^{d}} R(x) \EE \Big[\varrho(x)\varrho(0)\Big]dx,
\end{aligned}
\end{equation}
where $\mathcal{Q}_2$ defined in \eqref{e.defcalQ} is the two-point
covariance function of $\pi_\infty$, and $\varrho$ is sampled from
$\pi_\infty$. The last equality follows from spatial homogeneity of
the random field $\big(\rho(x)\big)_{x\in\bT^d}$. With a bit more effort, one can show the almost sure convergence of $\tfrac{1}{t}\log Z(t,x)$ to $\gamma_L(\beta)$; see \cite{KKM23} for the convergence in a stronger topology.

The $\gamma_L(\beta)$ is the Lyapunov exponent, or the limiting   free energy of the directed polymer. It  measures the exponentially growing speed of $Z(t,x)$, and describes, in the thermodynamic limit, the extent to which two paths sampled independently from the polymer measure overlap with each other.

In the general case, it is impossible to obtain an explicit expression for $\gamma_L(\beta)$, so it is natural to investigate its high temperature ($\beta\ll1$) behavior. The problem has been considered in the whole space. Denote the limiting free energy by $\gamma_\infty(\beta)$ in this
case, it has been shown in \cite{berger2017high,lacoin2010new,nakashima2014remark,nakashima2019free}, for a large class of discrete models and as $\beta\to0$,
\begin{equation}\label{e.inf24}
\tfrac{1}{\beta^4} \gamma_\infty(\beta)\to -\tfrac{1}{24},  \quad\quad \mbox{ in } d=1,
\end{equation}
and
\begin{equation}\label{e.infpi}
\beta^2\log \gamma_\infty(\beta)\to -\pi, \quad\quad \mbox{ in } d=2.
\end{equation}
The limiting constants $-\frac{1}{24}$ and $-\pi$ are universal as they do not depend on the specific distributions of the underlying random environment. For SHE on the whole line with a $1+1$ spacetime white noise, the Lyapunov exponent was computed as $\gamma_\infty(\beta)=-\tfrac{1}{24}\beta^4$ (see \cite{ACQ11}). In the periodic setting, one can pose the same question, and we have the following result \cite[Theorem 1.1, Proposition 4.1]{GK211}
\begin{theorem}\label{t.gamma}
Fix $L>0$. (i) For the case of $R(\cdot)=\delta(\cdot)$ in $d=1$, we have 
\begin{equation}\label{e.gammawhite}
\gamma_L(\beta)=-\tfrac{1}{2L}\beta^2-\tfrac{1}{24}\beta^4.
\end{equation}
(ii) For the case of $R\in C^\infty(\bT^d)$, we have the small $\beta\ll1$ expansion:
\begin{equation}\label{e.gammaex}
\gamma_L(\beta)=\gamma_L^{(2)}\beta^2+\gamma_L^{(4)}\beta^4+O(\beta^6),
\end{equation}
with 
\[
\gamma_L^{(2)}=-\tfrac{1}{2L^d}, \quad\quad \gamma_L^{(4)}=-\tfrac{1}{8\pi^2 L^{2d-2}}\sum_{0\neq n\in \Z^d} \tfrac{1}{|n|^2}\hat{R}^2\left(\tfrac{n}{L}\right),
\]
where $\hat{R}(\xi)
  {=\int_{\bT^d}R(x)e^{-i2\pi \xi\cdot x}dx}$.
\end{theorem}

The proof of \eqref{e.gammawhite} is a direct calculation, which we present below \footnote{The published version \cite[Equation (4.4)]{wrong} contains a mistake.}. In this case $\pi_\infty$ is explicitly given by \eqref{rho-i} in terms of a Brownian bridge. With $R(\cdot)=\delta(\cdot)$,  $d=1$  and $\varrho(x)=e^{\beta W(x)}/\int_0^L e^{\beta W(x')}dx'$, where $W$ is a standard Brownian bridge connecting $(0,0)$ and $(L,0)$, we have 
\[
\gamma_L(\beta)=-\tfrac12\beta^2 \int_0^L \EE \varrho(x)^2 dx=-\tfrac12\beta^2L\EE \varrho(0)^2= -\tfrac12\beta^2 L\EE(\int_0^L e^{\beta W(x)}dx)^{-2},
\]
where the second ``='' is due to the stationarity of the process
$\{\varrho(x)\}_{x\in\bT}$. The random variable
$\int_0^L e^{\beta W(x)}dx$ appears
frequently in physics and mathematical finance (see
\cite{MY05} for an extensive discussion).

\begin{proof}[Proof of \eqref{e.gammawhite}]
First, by the scaling property of the Brownian bridge
\[
\int_0^L e^{\beta W(x)}dx\stackrel{\text{law}}{=}L \int_0^1 e^{\beta\sqrt{L}W'(x)}dx,
\]
where $W'$ is a bridge connecting $(0,0)$ with $(1,0)$.
  To simplify the notation, define $Y_\lambda=\int_0^1 e^{\lambda W'(x)}dx$, so it remains to compute $\E Y_\lambda^{-2}$. Denote the density of $Y_\lambda$ by $f_\lambda(z)$, by \cite[Proposition 6.2,
p. 527]{Yor92} we have 
\begin{equation}
\label{fla}
f_\lambda (z)=\frac{4}{\pi
 \lambda^2 z^2}\exp\left\{-\frac{4}{\lambda^2 z}+\frac{2\pi^2}{ \lambda^2}\right\}\int_0^{\infty}\exp\left\{-\frac{2y^2}{ \lambda^2}-\frac{4\cosh
  y}{\lambda^2 z}
  \right\}\big(\sinh y\big)\sin\left(\frac{4\pi y}{\lambda^2}\right)dy.
\end{equation}
Using the above density formula, we have \begin{align}
\label{031006-21}
&
\E Y_\lambda^{-2}=\int_0^\infty z^{-2} f_\lambda(z)dz=\frac{4}{\pi
 \lambda^2 }\exp\left\{\frac{2\pi^2}{ \lambda^2}\right\}\int_0^{\infty}\exp\left\{-\frac{4}{\lambda^2 z} \right\}\frac{1}{
  z^4}\\
&
\times \left( \int_0^{\infty}\exp\left\{-\frac{2y^2}{ \lambda^2}-\frac{4\cosh
  y}{\lambda^2z}
  \right\}\big(\sinh y\big)\sin\left(\frac{4\pi y}{\lambda^2}\right)dy\right)dz.\notag
\end{align}
Changing variables $z'=(z\lambda^2)^{-1}$ we get
\begin{align}
\label{031006-21b}
&
\E Y_\lambda^{-2}=\frac{4 \lambda^4}{\pi
 }\exp\left\{ \frac{2\pi^2}{
  \lambda^2}\right\}\int_0^{\infty}z^2\exp\left\{- 4 z \right\} \\
&
\times \left(\int_0^{\infty}\exp\left\{-\frac{2y^2}{ \lambda^2}-4z\cosh
  y 
  \right\}\big(\sinh y\big)\sin\left(\frac{4\pi y}{\lambda^2}\right)dy \right)dz.\notag
\end{align}
Note that
\begin{align*}
\int_0^{\infty}z^2\exp\left\{- 4 z(1+\cosh y) \right\} d z=\frac{1}{32(1+\cosh y)^3}.
\end{align*}
Hence, we get
\begin{align}
\label{031006-21d}
\E Y_\lambda^{-2}&=\frac{ \lambda^4}{8 \pi
 }\exp\left\{ \frac{2\pi^2}{
  \lambda^2}\right\}  \int_0^{\infty}\frac{\sinh y}{(1+\cosh y)^3}\exp\left\{-\frac{2y^2}{ \lambda^2} 
  \right\}\sin\left(\frac{4\pi y}{\lambda^2}\right)dy.
\end{align}

Recall that
\[
\gamma_L(\beta)=-\frac{\beta^2}{2L}\E Y_{\beta\sqrt{L}}^{-2}.
\]
To complete the proof of \eqref{e.gammawhite} it remains to show that
$\E Y_{\lambda}^{-2}=1+\frac{\lambda^2}{12}$. To see the latter note
that, it follows from \eqref{031006-21d} that
\begin{align*}
\E Y_\lambda^{-2}=\frac{ \lambda^4}{16 \pi
 } {\rm Im}\left(\int_{\bbR}\frac{\sinh y}{(1+\cosh
                   y)^3}\exp\left\{-\frac{2(y -\pi i)^2}{ \lambda^2} 
                   \right\} dy\right).                
\end{align*}
Using the relations $\sinh (y+i\pi)=-\sinh y$, $\cosh
(y+i\pi)=-\cosh y$ and the change of variables $y':=y-i\pi$, we can further rewrite the right hand side as being
equal to 
\begin{align*}
& \frac{ \lambda^4}{16 \pi
 } {\rm Im}\left(\int_{\bbR-i\pi}\frac{\sinh y}{(\cosh
                   y-1)^3}\exp\left\{-\frac{2y^2}{ \lambda^2} 
                 \right\} dy\right)\\
  &
    =-\frac{ \lambda^4}{16 \pi
 } {\rm Im}\left(\int_{\bbR+i\pi}\frac{\sinh y}{(\cosh
                   y-1)^3}\exp\left\{-\frac{2y^2}{ \lambda^2} 
                   \right\} dy\right).                
\end{align*}
The last equality follows from the change of variables $y':=-y$
and the fact the integrand is odd. We can write therefore that
\begin{align*}
\E Y_\lambda^{-2}=\frac{ \lambda^4}{32 \pi
 } {\rm Im}\left(\int_{\cal C}\frac{\sinh z}{(\cosh
                   y-1)^3}\exp\left\{-\frac{2z^2}{ \lambda^2} 
                   \right\} dz\right).           
\end{align*}
Here ${\cal C}$ is an arbitrary counter-clockwise oriented, contour surrounding the unique pole $z=0$ of
the holomorphic function $\varphi(z):=  \frac{\sinh z}{(\cosh
                   z-1)^3}\exp\left\{-\frac{2z^2}{ \lambda^2} 
                   \right\}$.
Hence
\begin{align*}
\E Y_\lambda^{-2}=\frac{ \lambda^4}{16 
 } {\rm Re}\,\Big({\rm res}\,\varphi\Big),          
\end{align*}
where ${\rm res}\,\varphi$ is the residuum of the function $\varphi$ at
$0$. To compute the latter, note that for $|z|$ sufficiently small, we
can write
\begin{align*}
&(\cosh
                 z-1)^{-3}=\left(\frac{z^2}{2}\right)^{-3}\left[1+\frac{z^2}{12}+\frac{z^4}{360}+z^6\psi(z)\right]^{-3}=
\frac{8}{z^6}   \left[1-\frac{z^2}{4} +
    \frac{z^4}{30}  +z^6\psi(z)\right].
\end{align*}
Here and below $\psi(z)$ is some generic function holomorphic in a
               neighborhood of $0$. We have also used the expansion $
\frac{1}{(1+z)^3}=1-3z+6z^2+z^3 \psi(z)
$ valid for $|z|<1$.
We also have
 \begin{align*}             
  &
    \sinh z=z\left[1+\frac{z^2}{6}+\frac{z^4}{120}+z^6 \psi(z)\right],\\
  &
  \exp\left\{-\frac{2z^2}{\lambda^2}\right\}=1-\frac{2z^2}{\lambda^2}+  \frac{2z^4}{\lambda^4}+z^6 \psi(z).
 \end{align*}
 Putting all things together we get
 \begin{align*}             
  &\varphi(z) 
      =\frac{8}{z^5}\left\{1+z^2\left(\frac{1}{6}-\frac{2}{\lambda^2}\right)+z^4\left(\frac{1}{120}
    -\frac{1}{3 \lambda^2} +  \frac{2}{\lambda^4}\right)+z^6 \psi(z)\right\}\\
  &
    \times\left[1-\frac{z^2}{4} +
    \frac{z^4}{30}  +z^6 \psi(z)\right].
 \end{align*}
 We conclude therefore that
 $$
 {\rm res}\,\varphi=8\left( 
    \frac{1}{6 \lambda^2} +
                 \frac{2}{\lambda^4}  \right)
$$
and
$
\E Y_\lambda^{-2}=\frac{\lambda^2}{12 } +1$.    
The proof of \eqref{e.gammawhite}  is complete.
\end{proof}

The proof of \eqref{e.gammaex} is   based on a small $\beta\ll1$ expansion of the stationary PDE hierarchy \eqref{e.hierarchystationary}. Since $\int_{\bT^{nd}} \mathcal{Q}_n(\bx_{1:n})d\bx_{1:n}=1$, and  $\Delta \mathcal{Q}_n\approx 0$ at leading order, one immediately guesses  the leading order to be
\[
\mathcal{Q}_n(\bx_{1:n})\approx L^{-nd},\quad\quad \beta\ll1.
\] Assuming that 
\[
\mathcal{Q}_n=L^{-nd}+\beta^2\mathcal{Q}_{n,1}+\beta^4\mathcal{Q}_{n,2}+\ldots
\]
and plugging into \eqref{e.hierarchystationary}, we can solve
$\mathcal{Q}_{n,k}$  explicitly and recursively for all
$k\geq1$. Taking $n=2$ and using the formula \eqref{e.congamma}, we
obtain \eqref{e.gammaex}.
 
It is worth pointing out that the term $-\tfrac{1}{2L^d}\beta^2$ vanishes as $L\to\infty$, and is referred to as  the ``finite size correction'' in the physics literature \cite{krug}. It actually comes from the self-overlap of the uniform measure on $\bT^d$. After all, at $\beta=0$, the polymer measure becomes Brownian motion and the $\pi_\infty$, which is a probability measure on $\mathcal{M}_1(\bT^d)$, in this case degenerates to Dirac measure at the uniform measure on $\bT^d$.

\begin{remark}
We do not have  an analogue of  formula \eqref{e.congamma} in the whole space setting. In this case, the endpoint density $\rho(t,\cdot)$ spreads to infinity as $t$ grows, and there does not exist an apparent equilibrium. Nevertheless, the replica overlap 
\[
\mathcal{R}(\rho(t))=\int_{\R^{2d}} R(x_1-x_2)\rho(t,x_1)\rho(t,x_2)dx_1dx_2
\] is
invariant under the spatial shift of the polymer endpoint. By embedding the
endpoint distribution into an abstract space, which factors out the
spatial shift, significant progress has been made recently on
the localization properties of the endpoint distribution, see \cite{bakhtin2020localization,bates2020endpoint,broker2019localization}. In this case, the limiting free energy
can be expressed as the solution of a variational problem, generalizing \eqref{e.congamma} in a sense. There is a more recent result in the case of $R(\cdot)=\delta(\cdot)$ in $d=1$: for each $t>0$, the random density $\rho(t,\cdot)$ has a unique mode, denoted by $x_t$, and after a   shift by $x_t$, the following weak convergence on $C(\R)$ holds:  
 \begin{equation}\label{e.das}
 \{\rho(t,x_t+x)\}_{x\in\R} \Rightarrow  \left\{\frac{e^{-\B(x)}}{\int_{\R}e^{-\B(x')}dx'}\right\}_{x\in\R}, \quad\quad \mbox{ as } t\to\infty.
 \end{equation}
 Here $\B$ is a two-sided $3$d-Bessel process with diffusion coefficient $1$, see \cite[Theorem 1.5]{das} for more details. Thus, one may expect that the Lyapunov exponent in this case,  $\gamma_\infty(\beta)=-\tfrac{1}{24}\beta^4$, admits a probabilistic representation in terms of $\B$.
\end{remark}

\section{Fluctuations of the height function}\label{s.height}

In this section, we review our study on the random fluctuations of the height function $h(t,x)=\log Z(t,x)$. We first state the main result \cite[Theorem 1.1]{GK21} and \cite[Corollary 2.3]{ADYGTK22}:
\begin{theorem}\label{t.cltheight}
The height function $h(t,x)$ satisfies a central limit theorem: with $\gamma_L(\beta)$ given by \eqref{e.congamma}, for any $x\in\bT^d$, we have 
\begin{equation}\label{e.clt}
\tfrac{1}{\sqrt{t}}(h(t,x)-\gamma_L(\beta)t)\Rightarrow N(0,\sigma_L^2(\beta))
\end{equation}
in distribution as $t\to\infty$. The variance $\sigma_L^2(\beta)$ admits the following representation:
\begin{equation}\label{e.exvar}
\sigma_L^2(\beta)=\beta^2 \int_{\bT^{2d}}\EE\left[ \prod_{j=1}^2 \EE \left[\frac{e^{H_1(y_j)+H_2(y_j)}}{\int_{\bT^d}e^{H_1(y')+H_2(y')}dy'}\bigg|H_2\right] \right]R(y_1-y_2)dy_1dy_2,
\end{equation}
where $H_1,H_2$ are sampled independently from the invariant measure
for the KPZ equation. Alternatively, one can rewrite \eqref{e.exvar}
with $e^{H_1}, e^{H_2}$ replaced  by $\varrho_1,\varrho_2$ which are
sampled independently from the $\pi_\infty$ in Theorem~\ref{t.ofos}.

In the case of a $1+1$ spacetime white noise, we have 
\begin{equation}\label{e.exvarwhite}
\sigma_L^2(\beta)=\beta^2\EE \frac{\int_0^L e^{\beta (W_1(y)+W_2(y)+2W_3(y))}dy}{\int_0^L e^{\beta (W_1(y)+W_3(y))}dy\int_0^L e^{\beta(W_2(y)+W_3(y))}dy},
\end{equation}
where $\{W_i\}_{i=1,2,3}$ are independent standard Brownian bridges
connecting $(0,0)$ and $(L,0)$, and in this case,  for a fixed
  $\beta>0$  we  have 
\begin{equation}\label{e.decaysigma}
\sigma_L^2(\beta)\asymp L^{-1/2}, \quad\quad \mbox{ for } L\gg1.
\end{equation}
\end{theorem}

The central limit theorem in \eqref{e.clt} is not surprising. As
discussed in Section~\ref{s.projective}, the projective process $\rho$
mixes exponentially fast, and since one can write the height function
$\log Z$ as an additive functional of $(\rho(t,\cdot),\xi(t,\cdot))$,
as in \eqref{e.semilogZ}, it is only natural to expect Gaussian fluctuations. Given the fast mixing, what is not so clear is if the effective variance is strictly positive. This is usually a separate issue in the proof of the central limit theorem, regardless of how fast the decorrelation is, and we will deal with it using the special structure of the problem. 

While  the projective process $\rho$  was used in our argument as the underlying driving Markov process, one can also consider the solution to the stochastic Burgers equation $u=\nabla h=\nabla \log Z$. On a formal level, the equation \eqref{e.kpz} of $h$ has the forcing terms $u^2$ and $\xi$, and the solution is approximately an additive functional of $(u(t,\cdot),\xi(t,\cdot))$. Using the energy solution formulation,   \eqref{e.clt} can also be proved by utilizing the fast mixing of $u$, and this generalizes to a  class of fractional Burgers equations \cite{PY}.

What surprised us was that the effective diffusivity in \eqref{e.exvar} is expressed explicitly in terms of the invariant measure of the   projective process. 
It is indeed the case for the Lyapunov exponent $\gamma_L(\beta)$ 
in \eqref{e.congamma}, where we have written $\gamma_L(\beta)$ in terms of the invariant measure $\pi_\infty$. This is natural since $\gamma_L(\beta)$ represents the averaging growth speed of the interface,  which only depends on the ``one-time'' distribution of $\rho(t,\cdot)$ or $u(t,\cdot)$, depending if we use the semi-martingale decomposition in \eqref{e.semilogZ} or the equation \eqref{e.kpz} itself. However, for the variance, it should depend on the ``two-time'' correlation, and generally, the invariant measure alone is not sufficient. This indicates that the equation must possess a certain special structure to yield something like \eqref{e.exvar}. We will discuss this aspect in more detail in Section 5 below.

In the special case of a $1+1$ spacetime white noise, where the
invariant measure $\pi_\infty$ is known, the effective diffusivity can
be expressed in terms of three independent copies of  Brownian
bridges. This allows us to derive the asymptotics of
$\sigma_L^2(\beta)$ as $L\to\infty$ and obtain the decaying rate of
$\sigma_L^2(\beta)\asymp L^{-1/2}$ for a fixed
  $\beta>0$. When $L$ is fixed then, by a simple scaling argument one
  can conclude from \eqref{e.exvarwhite} and \eqref{e.decaysigma} that
  $$
\sigma_L^2(\beta)\asymp \beta^2, \, \mbox{ for } \beta\ll1\quad
\mbox{and}\quad \sigma_L^2(\beta)\asymp \beta^3, \, \mbox{ for } \beta\gg1.
$$

For general growth models in the cylindrical geometry, while the expressions of $\sigma_L^2(\beta)$ may vary depending on the model, the decaying rate of $L^{-1/2}$ is expected to be universal and related to the scaling exponents in the KPZ universality class. The results on the exclusion process can be found in \cite{DEM93}. For the KPZ equation with $1+1$ spacetime white noise, take $x=0$ in  \eqref{e.clt}, we have $\Var\, [h(t,0)]\approx \sigma_L^2(\beta)t\approx L^{-1/2}t$. On the other hand, consider the ``transversal roughness'' $\Var [h(t,x)-h(t,0)]$, and assuming equilibrium so that $h(t,x)-h(t,0)$ is approximately a Brownian bridge, we get $\Var [h(t,x)-h(t,0)] \asymp |x|$. To observe nontrivial correlations between $h(t,0)$ and $h(t,x)$ with $x\in [0,L]$, it is natural to expect the one point variance to be of the same order as the transversal roughness. Thus, the transition occurs in the regime of $tL^{-1/2}\asymp L$. This leads to the critical length scale $L\asymp t^{2/3}$ and the order of fluctuations $\sqrt{\Var h(t,0)}\asymp t^{1/3}$.

The following result makes the above heuristics rigorous in certain regimes as $t,L\to\infty$ together \cite[Theorem 1.1]{ADYGTK22}: 
\begin{theorem}\label{t.123}
	 In the case of $1+1$ spacetime white noise, assume $h(t,x)$ starts from equilibrium, 	 and suppose that $L=\lambda t^{\alpha}$ for some
         $\lambda>0$ and $\alpha\geq 0$. Then, there is a universal constant $\delta>0$ so that the following
        holds:
	\begin{equation}\label{e.conjecture1}
		\Var [h(t,x)] \asymp
		\begin{cases}
			t^{1-\frac{\alpha}{2}}, & \text{if }\alpha\in[0,2/3)\text{ and }\lambda\in(0,\infty);                             \\
			t^{2/3},        & \text{if }\alpha = 2/3\text{ and }\lambda\in(0,\delta).
		\end{cases}
              \end{equation}
In the general case of $\alpha=2/3$ and $\lambda\in(0,\infty)$,  the upper bound holds 
\[
\Var[ h(t,x)]\les  t^{2/3}.
\] 
\end{theorem}

In the following two sections, we will present and discuss two different proofs of the central limit theorem \eqref{e.clt}. 
Section~\ref{s.homogenization} covers the first proof, which is based
on the semi-martingale decomposition, discussed in
Section~\ref{s.projective}, and the homogenization argument for an
additive functional of the projective
process. Section~\ref{s.clarkocone} presents a second proof, which
utilizes tools from the Malliavin calculus to derive the explicit variance in \eqref{e.exvar}. We will not delve into the proof of \eqref{e.decaysigma}, which, although it only involves calculations with Brownian bridges and may appear straightforward, is actually quite complex. Interested readers should refer to \cite[Section 4]{ADYGTK22} for details. Given the decay of the variance $\sigma_L^2(\beta)\asymp  L^{-1/2}$, the proof of Theorem 3.2 is relatively straightforward and can be found in \cite[Section 3]{ADYGTK22}.

\subsection{Homogenization of additive functional of projective process}
\label{s.homogenization}

Let us recall that, using \eqref{e.semilogZ} and \eqref{e.logZtx}, one can write 
\begin{equation}\label{e.decomh}
\begin{aligned}
h(t,x)&=\log Z(t,x)=\log \bar{Z}_t+\log \rho(t,x)\\
&=\log \bar{Z}_0+\beta \int_0^t\int_{\bT^d} \rho(s,x)\xi(s,x)dxds-\tfrac12\beta^2 \int_0^t \mathcal{R}(\rho(s))ds+\log \rho(t,x).
\end{aligned}
\end{equation}
 Let us take a look at the terms on the r.h.s. of the above equation. For $t\gg1$, by \eqref{e.mmbdrho}, we know that $\log \rho(t,x)$ is of order $O(1)$, hence it vanishes in large time after divided by $\sqrt{t}$. 
The martingale term is straightforward to deal with, appealing to the martingale central limit theorem and the one-force-one-solution principle established for the $\rho$ process in Theorem~\ref{t.ofos}. We sketch the proof for general martingales of the form 
\begin{equation}\label{e.generalmartingale}
\cM_t=\int_0^t\int_{\bT^d} F(\rho(s),y)\xi(s,y)dyds,
\end{equation}
where $F:D^\infty(\bT^d)\times \bT^d\to\R$ is a functional
  that is sufficiently regular (e.g. H\"older continuous) in the first
variable and  satisfies
a certain growth condition. Recall that $D^\infty(\bT^d)$ is the space
of continuous densities w.r.t. the Lebesgue measure, equipped with the
topology of uniform convergence. By the martingale central limit theorem,
see e.g. \cite[Theorem VIII.2.17]{JS03},
to show $\eps \cM_{t/\eps^2}$ converges in distribution in
$C[0,\infty)$, as $\eps\to0$, to some Brownian motion with variance
$\Sigma^2>0$, it suffices to establish that the process $\eps^2\la
\cM\ra_{t/\eps^2}\to \Sigma^2 t$ in probability in
$C[0,\infty)$. 
The rescaled bracket process takes the form 
\begin{equation}\label{e.bracketcM}
\eps^2\la \cM\ra_{t/\eps^2}=\eps^2\int_0^{t/\eps^2}\int_{\bT^{2d}} F(\rho(s),y_1)F(\rho(s),y_2)R(y_1-y_2)dy_1dy_2ds.
\end{equation}
The tightness of $\{\eps\la \cM\ra_{t/\eps^2}:
t\geq0\}_{\eps\in(0,1)}$ can be directly checked by the Kolmogorov
criterion, using the appropriate growth assumption on $F(\cdot)$ and
the moment bound on $\rho$ given by \eqref{e.mmbdrho}. To show the
convergence of finite dimensional distributions, it is enough to prove that for fixed $t>0$, $\eps^2\la \cM\ra_{t/\eps^2}\to \Sigma^2 t$ in probability.  There are two possible approaches to addressing this: (i) one can show the stronger $L^2(\Omega)$ convergence by proving that 
\[
\EE \Big[\eps^2\la \cM\ra_{t/\eps^2}\Big]\to \Sigma^2t, \quad\mbox{ and } \quad \EE \Big[ (\eps^2\la \cM\ra_{t/\eps^2})^2 \Big]\to (\Sigma^2t)^2.
\]
This was the proof presented in \cite[Proposition 4.9, Proposition 5.27]{GK21}. (ii)  One can first replace the $\rho(s)$ in \eqref{e.bracketcM} by its stationary version $\tilde{\rho}(s)$. In other words, rather than starting from $\rho(0)$, we let $\tilde{\rho}$ start from the invariant measure and experience the same noise $\xi$ as $\rho$. The error induced by this approximation is small, by the exponential contraction in \eqref{051705-23} and the factor $\eps^2$ on the r.h.s. of \eqref{e.bracketcM}. Since $\tilde{\rho}$ starts from stationarity, one can apply the  Birkhoff’s Ergodic Theorem to conclude the almost sure convergence of 
\[
\eps^2\int_0^{t/\eps^2}\int_{\bT^{2d}} F(\tilde{\rho}(s),y_1)F(\tilde{\rho}(s),y_2)R(y_1-y_2)dy_1dy_2ds,
\]
which would complete the proof.

To summarize, with   Theorem~\ref{t.ofos}, it is relatively straightforward to prove an invariance principle for general martingales of the form \eqref{e.generalmartingale}. Thus, the remaining task is to  approximate the additive functional in \eqref{e.decomh}
\begin{equation}\label{e.additivefunctional}
\int_0^t\mathcal{R}(\rho(s))ds
\end{equation}
by a martingale, which is a classical problem in probability
\cite{KV86,KLO12}.  The idea is explained in the following section.

\subsubsection{Poisson equation, corrector, martingale decomposition}
\label{s.poissoncorrector}

By the exponential contraction in \eqref{051705-23},
it is rather standard to construct the solution to the Poisson equation 
\begin{equation}\label{e.poisson}
-\mathcal{L}\chi=\tilde{\cR},
\end{equation}
with $\mathcal{L}$ the generator of the projective process, and
$\tilde{\cR}$ the properly centered $\cR$:  
\begin{equation}\label{e.centerR}
\tilde{\cR}=\cR-\int_{D^\infty(\bT^d)} \cR(v)\pi_\infty(dv),
\end{equation}
where $D^\infty(\bT^d)$ is the space of continuous densities w.r.t. the Lebesgue measure
on $\bT^d$.
More precisely, the solution to \eqref{e.poisson} can be formally written as
\begin{equation}\label{e.integralcorrector}
\chi(v)=\int_0^\infty \mathcal{P}_t \tilde{\cR}(v)dt,
\end{equation}
where we chose e.g. $v\in D^\infty(\bT^d)$, and $\mathcal{P}_t$ is the semigroup of the process $\rho$
\[
\mathcal{P}_t \tilde{\cR}(v)=\EE [\tilde{\cR}(\rho(t))|\rho(0)=v].
\]
To make sense of the integral representation of $\chi$ in \eqref{e.integralcorrector}, one needs to show that $\mathcal{P}_t \tilde{\cR}(v)$ is sufficiently small for $t\gg1$: 
\begin{equation}\label{e.Ptsmall}
|\mathcal{P}_t \tilde{\cR}(v)|\leq Ce^{-\lambda t}.
\end{equation}
To show the above estimate, again, the idea is to approximate $\rho$ by $\tilde{\rho}$, which starts from the invariant measure and experiences the same noise $\xi$ as $\rho$. Then \eqref{051705-23} implies that 
\[
|\EE [\tilde{\cR}(\rho(t))|\rho(0)=v]-\EE \E_{\pi_\infty}[\tilde{\cR}(\tilde{\rho}(t))]| \leq Ce^{-\lambda t},
\]
where $\E_{\pi_\infty}$ is the expectation on $\tilde{\rho}(0)$. On the other hand, by the invariance of $\pi_\infty$ and the centering in \eqref{e.centerR}, we have $\EE \E_{\pi_\infty}[\tilde{\cR}(\tilde{\rho}(t))]=0$. This completes the proof of \eqref{e.Ptsmall}. More details can be found in \cite[Lemma 5.1]{GK21}.

We define the $\chi$, given by integral on the r.h.s. of
\eqref{e.integralcorrector}, as the solution to the Poisson
equation. It is referred to as the {\em corrector}, an object used
extensively in the homogenization theory. 

With the construction of the corrector $\chi$, the next step is to
write down the Dynkin's martingale associated with it, with the drift
term equal  to the additional functional $\int_0^t
\tilde{\cR}(\rho(s))ds$ we started with. This is not trivial,
especially when $\xi$ is the $1+1$ spacetime white noise. In this
case, the SPDE  \eqref{e.spderho} satisfied  by $\rho$ is only formal,
and, due to the singular noise, $\rho=Z/\bar Z$  is not a strong solution to this equation. In \cite[Section 5.3]{GK21}, we dealt with the technicalities through an approximation by a smooth noise, and in the end, the additive functional can be written as \cite[Corollary 5.20]{GK21}
\begin{equation}\label{e.deadditive}
\int_0^t \tilde{\cR}(\rho(s))ds=\chi(\rho(0))-\chi(\rho(t))+\beta \int_0^t \int_{\bT^d} \D\chi(\rho(s),y)\rho(s,y)\xi(s,y)dyds,
\end{equation}
for any $\rho(0)\in D^\infty(\bT^d)$. Here $\int_0^t
\D\chi(\rho(s),y)\rho(s,y)\xi(s,y)dyds$ is the martingale we were
looking for in the first place, and $\D\chi$ is the Frechet gradient
of the corrector, defined and studied in detail in \cite[Section
5.2]{GK21}. One can formally obtain the expression of the martingale
through an application of the It\^o formula to $\chi(\rho(t))$,
together with the (formal) SPDE \eqref{e.spderho} satisfied by
$\rho$. Note that there are two martingale terms on the r.h.s. of
\eqref{e.spderho}, while the second one does not contribute  in
\eqref{e.deadditive}. This is due to the fact that
\begin{equation}
  \label{010908-24}
  \int_{\bT^d}\D\chi(\rho(s),y)\rho(s,y) dy=0,
\end{equation}
see \cite[Proposition 5.14]{GK21}.

Now one can combine \eqref{e.decomh} and \eqref{e.deadditive} to rewrite the height function as
\begin{equation}\label{e.deadditive2}
\begin{aligned}
h(t,x)-\gamma_L(\beta)t=&\log \bar{Z}_0+\log \rho(t,x)-\tfrac12\beta^2\chi(\rho(0))+\tfrac12\beta^2\chi(\rho(t))\\
&+\beta\int_0^t\int_{\bT^d}\rho(s,y)[1-\tfrac12\beta^2\D\chi(\rho(s),y)]\xi(s,y)dyds.
\end{aligned}
\end{equation}
Dividing by $\sqrt{t}$ on both sides,  all the terms on the first line of the r.h.s. vanish as $t\to\infty$. So it remains to study the martingale term
\[
\tfrac{1}{\sqrt{t}}\mathscr{M}_t=\tfrac{\beta }{\sqrt{t}} \int_0^t\int_{\bT^d}\rho(s,y)[1-\tfrac12\beta^2\D\chi(\rho(s),y)]\xi(s,y)dyds.
\]
By our previous discussion on \eqref{e.generalmartingale}, we conclude that
\begin{equation}\label{e.conmartingale}
\tfrac{1}{\sqrt{t}}\mathscr{M}_t\Rightarrow N(0,\sigma_L^2(\beta)),
\end{equation}
with the following expression of the effective diffusivity
\begin{equation}\label{e.sigmaex1}
\sigma_L^2(\beta)=\beta^2\EE \int_{\bT^{2d}}\prod_{j=1}^2 \varrho(y_j)[1-\tfrac12\beta^2\D\chi(\varrho,y_j)]R(y_1-y_2)dy_1dy_2,
\end{equation}
with $\varrho$ sampled from $\pi_\infty$. This completes the proof of
\eqref{e.clt}. With the homogenization type approach, it is  expected that the effective diffusivity is expressed in terms of the gradient of the corrector, which is indeed the case in \eqref{e.sigmaex1}.

It remains to show that the $\sigma_L^2(\beta)$ obtained above is
strictly positive, otherwise, the convergence in
\eqref{e.conmartingale} might not be  fully
informative, and adjusting the scaling could be necessary. Since the
size of the variance plays a crucial role when enlarging the torus in
 our later exposition, we sketch the proof of
the strict positivity of $\sigma_L^2(\beta)$ below. As it will be
shown, using only soft information on the corrector, it is impossible
to obtain sharp bounds  in terms of the torus size.   The integral 
\[
\int_{\bT^{2d}}\prod_{j=1}^2 \varrho(y_j)[1-\tfrac12\beta^2\D\chi(\varrho,y_j)]R(y_1-y_2)dy_1dy_2
\]
can be viewed as a weighted $L^2(\bT^d)$ norm of the function
$f(y):=\varrho(y) [1-\tfrac12\beta^2\D\chi(\varrho,y)]$. One can write
it using   its Fourier coefficients
\[
\int_{\bT^{2d}}f(y_1)f(y_2)R(y_1-y_2)dy_1dy_2= \sum_{k\in \Z^d}|\hat{f}(k)|^2\hat{R}(k).
\] 
Here the Fourier coefficients are defined as
  $\hat{g}(k)=\tfrac{1}{L^{d/2}}\int_{\bT^d}g(x)e^{-i2\pi k\cdot
    x/L}dx$ for any   $g\in L^1(\bT^d)$. Use the assumption
$\int_{\bT^d} R(x)dx=1$ and the fact that   (see \eqref{010908-24})
\[
\int_{\bT^d} f(y)dy=\int_{\bT^d} \varrho(y) [1-\tfrac12\beta^2\D\chi(\varrho,y)]dy=\int_{\bT^d} \varrho(y)dy=1,
\]
we derive
\begin{equation}\label{e.varlowbd}
\sigma_L^2(\beta)=\beta^2\EE  \sum_{k\in \Z^d}|\hat{f}(k)|^2\hat{R}(k)\geq \beta^2 |\hat{f}(0)|^2\hat{R}(0)=\beta^2L^{-d}.
\end{equation}
Therefore, in all dimensions, using the above argument we have a lower
bound of order $L^{-d}$. It is interesting to note that the
$\beta^2L^{-d}$ on the r.h.s. of \eqref{e.varlowbd} is precisely the
asymptotic variance for the Edwards-Wilkinson equation on a torus of
size $L$: $\partial_t \mathcal{U}=\tfrac12\Delta\mathcal{U}+\beta
\xi$. In $d=1$, by \eqref{e.decaysigma}, we know that
$\sigma_L^2(\beta)\asymp L^{-1/2}\gg L^{-1}$. It is unclear to us how
to use the corrector and the expression for the $\sigma_L^2(\beta)$ in \eqref{e.sigmaex1} to obtain the right decaying rate of $\sigma_L^2(\beta)$.

\subsection{Clark-Ocone formula}
\label{s.clarkocone}

In this section, we present a different proof of the central limit
theorem for the height function, which relies on a   formula from Malliavin calculus. The Clark-Ocone formula allows us to express a (centered) random variable as an explicit It\^o integral. Here is an informal statement (see e.g. \cite[Section 6.1]{CKNP19} for a precise statement): for a smooth random variable $X$ defined on $(\Omega,\F,\PP)$, we have 
\begin{equation}\label{e.clarkocone}
X-\EE X=\int_0^\infty\int_{\bT^d} \EE [D_{s,y}X|\F_s]\xi(s,y)dyds,
\end{equation}
where $D_{s,y}$ is the Malliavin derivative operator associated with the noise $\xi$ on $\R_+\times \bT^d$, $(\F_s)_{s\geq0}$ is the natural filtration generated by $\xi$, and the above integral is interpreted in the It\^o sense.

The key aspect  of the above formula is that the r.h.s. is in the form of an It\^o integral, so,  when viewed as a sum of martingale differences, there is no correlation among the summands. The way this formula helps ``absorbing'' the correlation  is different from the homogenization argument used in Section~\ref{s.poissoncorrector}. There, to study the additive functionals of the projective process, we employed the corrector to ``absorb'' the temporal correlations so that  the variance of the additive functional  can be written in terms of the gradient of the  corrector, and eventually we obtained 
\[
\Var h(t,x)\approx t\sigma_L^2(\beta)
\]
with   the variance $\sigma_L^2(\beta)$ given by \eqref{e.sigmaex1},  in terms of the gradient of the corrector.  With \eqref{e.clarkocone}, one  applies it to $h(t,x)$ and use  It\^o isometry to derive
\begin{equation}\label{e.varhco}
\begin{aligned}
\Var\, h(t,x)=\int_0^t\int_{\bT^{2d}}& \EE\big[\EE [D_{s,y_1} h(t,x)|\F_s]\EE [D_{s,y_2}h(t,x)|\F_s]\big]\\
&\times  R(y_1-y_2)dy_1dy_2ds.
\end{aligned}
\end{equation}

Whether the above expression is useful or not for our purpose obviously depends on how explicitly we can compute the conditional expectation $\EE[ D_{s,y} h(t,x)|\F_s]$. For the KPZ equation, it turns out that $D_{s,y}h(t,x)$ has the interpretation as the midpoint density of a directed polymer. This can be readily seen in the discrete setting, where the   random environment is represented by a sequence of i.i.d. random variables, and $D_{s,y}h(t,x)$ becomes the derivative of the free energy with respect to  the random variable at the spacetime point $(s,y)$, which is precisely the quenched probability of the polymer path passing through $(s,y)$. In our continuous setting, we have
\[
D_{s,y}h(t,x)=D_{s,y} \log Z(t,x)=\frac{D_{s,y}Z(t,x)}{Z(t,x)}.
\]
The calculation of $D_{s,y}Z(t,x)$ can  be done either on the level of the mild formulation of the SHE \cite[Theorem 3.2]{CKNP20} or through the formal Feynman-Kac formula and an approximation argument: 
\begin{equation}\label{e.madeZ}
D_{s,y}Z(t,x)=\beta \G_{t,s}(x,y)Z(s,y),
\end{equation}
where $\G_{t,s}(x,y)$ is the Green's function of the SHE defined in
\eqref{e.defgreen}  and $Z(s,y)$ solves \eqref{e.she}.  Using the forward/backward polymer endpoint density, one can rewrite $D_{s,y}h(t,x)$ as 
\begin{equation}\label{e.madeh}
\begin{aligned}
D_{s,y}h(t,x)&=\beta \frac{\G_{t,s}(s,y)Z(s,y)}{Z(t,x)}\\
&=\beta \frac{\G_{t,s}(x,y)\int_{\bT^d} \G_{s,0}(y,w)Z_0(w)dw}{\int_{\bT^{2d}} \G_{t,s}(x,y')\G_{s,0}(y',w)Z_0(w)dwdy'}\\
&=\beta \frac{\rhob(t,x;s,y)\rhof(s,y;0,\mu_{Z_0})}{\int_{\bT^d}\rhob(t,x;s,y')\rhof(s,y';0,\mu_{Z_0})dy'}.
\end{aligned}
\end{equation}
Here $\mu_{Z_0}$ is the measure on $\bT^d$ with the density $Z_0 /\bar Z_0$. In other words, modulo the parameter $\beta$,   $D_{s,y}h(t,x)$ is precisely the midpoint density of the directed polymer starting at $(t,x)$, running backwards in time, with the boundary condition $\mu_{Z_0}$.

\begin{remark}
To see why \eqref{e.madeZ} holds on a formal level, we compute the derivative on the level of SHE. Using the product rule and the formal identity $D_{s,y}\xi(t,x)=\delta(t-s,x-y)$, we obtain
\[
\begin{aligned}
\partial_t D_{s,y}Z(t,x)&=\frac12\Delta_x D_{s,y}Z(t,x)+D_{s,y}(\beta \xi(t,x)Z(t,x))\\
&=\frac12\Delta_x D_{s,y}Z(t,x)+\beta \xi(t,x)D_{s,y}Z(t,x)+\beta \delta(t-s,x-y)Z(t,x)\\
&=\frac12\Delta_x D_{s,y}Z(t,x)+\beta \xi(t,x)D_{s,y}Z(t,x)+\beta
\delta(t-s,x-y)Z(s,y),\quad t>0\\
 D_{s,y}Z(0,x)& =0.
\end{aligned}
\]
In other words, for fixed $(s,y)$,  with $0<s<t$,
the function $f(t,x):=D_{s,y}Z(t,x)$ solves the SHE with  the zero
initial data and the extra
kick forcing term $\beta \delta(t-s,x-y)Z(s,y)$. Hence $f(t,x)=0$ for
$0\le t<s$ and
\[
\begin{aligned}
\partial_t f(t,x)&=\frac12\Delta f(t,x)+\beta \xi(t,x)f(t,x), \quad\quad t>s, \\
f(s,x)&=\beta \delta(x-y)Z(s,y).
\end{aligned}
\]
Using the Green's function of SHE, $f$ can be written as $f(t,x)=\beta
\G_{t,s}(x,y)Z(s,y)$, for $t>s$, so we have   \eqref{e.madeZ}. This heuristic derivation is pretty formal but it shows us essentially what happens on the level of equations \cite[Theorem 3.2]{CKNP20}.
\end{remark}

Using \eqref{e.madeh}, we can almost immediately identify the origin of  the variance expression in \eqref{e.exvar}. For $s\gg1$ and $t-s\gg1$, by the exponential stabilization of the $\rho$ process, the following approximation holds
\begin{equation}\label{e.appmid}
\begin{aligned}
D_{s,y}h(t,x)&=\beta \frac{\rhob(t,x;s,y)\rhof(s,y;0,\mu_{Z_0})}{\int_{\bT^d}\rhob(t,x;s,y')\rhof(s,y';0,\mu_{Z_0})dy'}\\
&\approx \beta  \frac{\rhob(t,\varrho_1;s,y)\rhof(s,y;0,\varrho_2)}{\int_{\bT^d}\rhob(t,\varrho_1;s,y')\rhof(s,y';0,\varrho_2)dy'}=:\beta \rho_{\mathrm{m}}(t;s,y),
\end{aligned}
\end{equation}
where $\varrho_1,\varrho_2$ are sampled independently from the
invariant measure of the $\rho$ process, which are also independent of
the noise $\xi$. The error in the above approximation is bounded from
above by 
$e^{-\lambda s}+e^{-\lambda(t-s)}$, up to a multiplicative constant,  in light of \eqref{051705-23}. The function
$\rho_{\mathrm{m}}(t;s,\cdot)$ can be interpreted as the midpoint
density of the directed polymer, with the starting and ending points
distributed according to $\varrho_1,\varrho_2$ respectively. Since
$\varrho_1,\varrho_2$ are sampled  independently
from the invariant measure and the forward and backward polymer
density processes are based on the noises  in the time intervals
$[0,s]$ and  $[s,t]$ respectively, the  $C(\bT^d)-$valued random
variables $\rhof(s,y;0,\varrho_2)$ and $\rhob(t,\varrho_1;s,y)$ are independent. 
In addition,  for any $t\geq s\geq0$ and as $C(\bT^d)-$valued random
variables satisfy 
\[
\rhob(t,\varrho_1;s,y)\stackrel{\text{law}}{=} \varrho_1(y), \quad\quad \rhof(s,y;0,\varrho_2)\stackrel{\text{law}}{=} \varrho_2(y).
\]
Therefore,
\[
\rho_{\mathrm{m}}(t;s,y)\stackrel{\text{law}}{=}\frac{\varrho_1(y)\varrho_2(y)}{\int_{\bT^d}\varrho_1(y')\varrho_2(y')dy'}.
\]
Since there are multiple sources of randomnesses here, including the noise $\xi$ and the boundary conditions $\varrho_1,\varrho_2$, to deal with the conditional expectation $\EE [D_{s,y}h(t,x)|\F_s]$, one needs to average out different sources of randomnesses with some care. To avoid the confusion, in the following we use $\EE$ as the expectation only on $\xi$, and $\E_i$ as the expectation with respect to $\varrho_i$. 

First,  for any $t\geq s\geq0$, we have the identity in law
\begin{equation}\label{e.identitylaw}
\E_1\EE [\rho_{\mathrm{m}}(t;s,y)|\F_s]\stackrel{\text{law}}{=}\E_1\left[\frac{\varrho_1(y)\varrho_2(y)}{\int_{\bT^d}\varrho_1(y')\varrho_2(y')dy'} \right].
\end{equation}
This is because the noise is white in time, so $\EE[\cdot|\F_s]$ can be treated as the average of    the noise in the time interval $[s,t]$, and to compute $\E_1\EE [\rho_{\mathrm{m}}(t;s,y)|\F_s]$, it is equivalent with averaging out the factor $\rhob$ in \eqref{e.appmid}.

Secondly, by the Clark-Ocone formula and the approximation in \eqref{e.appmid}, we have
\[
h(t,x)-\EE h(t,x)\approx \beta \int_0^t\int_{\bT^d} \EE [\rho_{\mathrm{m}}(t;s,y)|\F_s]\xi(s,y)dyds,
\]
with the error of $O(1)$ in $L^2(\Omega)$. To replace the approximation ``$\approx$'' with an identity, one can consider a modified partition function in which both the starting and ending points of the directed polymer are at equilibrium. For further details, refer to \cite[Section 2]{ADYGTK22}, as well as the discussion in Section~\ref{s.corrector} below, particularly equation~\eqref{e.deadditive4}. Taking $\E_1$ on both sides, applying It\^o isometry and then taking $\E_2$, we further derive 
\[
\Var \Big[h(t,x)\Big]\approx \beta^2\int_0^t\int_{\bT^{2d}} \E_2 \EE \left[ \prod_{j=1}^2 \E_1 \EE [\rho_{\mathrm{m}}(t;s,y_j)|\F_s] \right] R(y_1-y_2)dy_1dy_2ds.
\]
By \eqref{e.identitylaw}, the r.h.s. equals to 
\[
\begin{aligned}
&\beta^2\int_0^t\int_{\bT^{2d}} \E_2   \left[ \prod_{j=1}^2 \E_1\left[\frac{\varrho_1(y_j)\varrho_2(y_j)}{\int_{\bT^d}\varrho_1(y')\varrho_2(y')dy'} \right]\right] R(y_1-y_2)dy_1dy_2ds\\
&=\beta^2 t \int_{\bT^{2d}} \E_2   \left[ \prod_{j=1}^2 \E_1\left[\frac{\varrho_1(y_j)\varrho_2(y_j)}{\int_{\bT^d}\varrho_1(y')\varrho_2(y')dy'} \right]\right] R(y_1-y_2)dy_1dy_2.
\end{aligned}
\]
%
Since one can write $\varrho_i$ as $\varrho_i(\cdot)=e^{H_i(\cdot)}/\int_{\bT^d} e^{H_i} $, with $H_i$ sampled from the invariant measure for the KPZ equation, we derive the variance formula in \eqref{e.exvar}.

The above discussion is almost a complete proof of \eqref{e.exvar},
except that one needs to justify the approximation rigorously. Given
\eqref{051705-23} and \eqref{e.mmbdrho}, this is fairly
straightforward in the region of $1\leq s\leq t-1$. For
$s\in[0,1]\cup[t-1,t]$, one can analyze it separately.

 To  
 prove the    Gaussian fluctuations from the approximation  
\begin{equation}\label{e.deadditive1}
\  h(t,x)- \EE h(t,x)\approx \beta \int_0^t\int_{\bT^d} \E_1\EE [\rho_{\mathrm{m}}(t;s,y)|\F_s]\xi(s,y)dyds,
\end{equation}
it is not hard to notice that the r.h.s. is actually a martingale of the form \eqref{e.generalmartingale}, thus, one can apply the martingale central limit theorem and the Birkhoff’s Ergodic Theorem to complete the proof.

To summarize, in the periodic setting, it is relatively straightforward to apply the Clark-Ocone formula to study the height fluctuations, leveraging the fast mixing the projective process as established by  Theorem~\ref{t.ofos}. The same proof, along with the mixing result from \cite{Pa22}, extends to the open KPZ equation, leading to a central limit theorem with a variance formula similar to \eqref{e.exvar}. The effectiveness of   this argument stems from the special structure of the Malliavin derivative $D_{s,y}h(t,x)$ as in \eqref{e.madeh}, which allows  it to be  approximated by the stationary midpoint density $\rho_{\mathrm{m}}(t;s,y)$ in the large time limit.

\begin{remark}
 
 It is worth mentioning that the KPZ type problems exhibit a
 phenomenon known as super-concentration \cite{Cha14}, with the
 variance of  the quantity  of interest growing sublinearly with respect to the size of the system. Standard concentration estimates, such as the Efron-Stein or Gaussian-Poincar\'e inequalities, typically yield variance bounds that are linear in the size of the system. From a technical point of view, one can  perceive the variance identity \eqref{e.varhco} as a strict improvement of 
 the Gaussian-Poincar\'e inequality. How to use it in other context
 seems to be a question  of great interest.
 \end{remark}

\section{Fluctuations of the winding number}
\label{s.winding}

Another physically interesting quantity in the periodic setting is the winding number of the directed polymer on a cylinder, which refers to the algebraic number of turns the polymer path makes around the cylinder. This problem is equivalent to studying the fluctuations of the endpoint of a directed polymer in a random and spatially periodic environment. We are particularly interested in the statistical fluctuations of the winding number for a very long polymer path. The main result in this section is a central limit theorem for this quantity. In the case of 
 $1+1$ spacetime white noise, we derive an explicit expression for the effective diffusivity, akin to \eqref{e.exvarwhite}. To simplify the presentation, throughout this section, we assume the length of the torus $L=1$. The result for general $L$ can be obtained through scaling.
 
 To state the result, we need to introduce more notations. Recall that $Z$ solves the SHE on $\R_+\times \bT^d$ with the noise $\xi$. To define the winding number properly, we consider a periodic extension of $\xi$ so that it is defined on $\R_+\times \R^d$, and let $\sfZ$ be the solution to 
\begin{equation}\label{e.defsfZ}
\begin{aligned}
\partial_t \sfZ(t,x)&=\frac12\Delta \sfZ(t,x)+\beta \xi(t,x) \sfZ(t,x), \quad\quad t>0, x\in\R^d,\\
\sfZ(0,x)&=\delta(x).
\end{aligned}
\end{equation}
The choice of the initial data $\delta(x)$ is by no means essential --
it leads to a point-to-line polymer measure, which is standard in the
literature. Compared to \eqref{e.she}, the main difference is that $Z$
is defined  there  on $\R_+\times \bT^d$, so that
one can think of  the initial data $Z(0,\cdot)$ as being automatically
periodic, while $\sfZ$ in \eqref{e.defsfZ} is defined on $\R_+\times \R^d$, with the non-periodic initial data. Using $\sfZ$, we can define the quenched endpoint density of the polymer (on $\R^d$ rather than $\bT^d$) as 
\begin{equation}\label{e.defsfp}
\sfp(t,x)=\frac{\sfZ(t,x)}{\int_{\R^d} \sfZ(t,x')dx'},\quad\quad t>0,x\in\R^d,
\end{equation}
and one can think of the winding number of the polymer path of length
$t$ as a random variable sampled from the (random)
   density $\sfp(t,\cdot)$, hence it inherits both the
randomness from $\xi$ and the sampling. Note that in   dimensions
$d\ge 2$, the winding number  is actually not
a number but a vector, so, to be more precise, we are studying the
displacement of the polymer endpoint, and with some abuse of
terminology,  we  refer to it as the winding number. 

By the formal use of the Feynman-Kac formula,   one can write $\sfp$ as 
\[
\sfp(t,x)=\frac{\E_B \exp(\beta\int_0^t\xi(s,B_s)ds)\delta(B_t-x)}{\E_B \exp(\beta\int_0^t\xi(s,B_s)ds)}
\]
where $B$ is a standard Brownian motion starting from the origin.

 One should note the difference between the $\sfp(t,x)$ above and the $\rho(t,x)$ defined in \eqref{e.rho} before. Suppose the initial data of $Z(0,x)=Z_0(x)$ is chosen as the Dirac function on $\bT^d$, then $Z(t,\cdot)$ can be viewed as the periodization of $\sfZ(t,\cdot)$, and hence $\rho$ is the periodization of $\sfp$:
\[
\rho(t,x)=\sum_{j\in \Z^d} \sfp(t,x+j).
\]

Now we can state the main result \cite[Theorem 1.1]{YGTK22} and \cite[Theorem 1.1]{YGTK23}:
\begin{theorem}\label{t.winding}
The polymer endpoint satisfies a central limit theorem in the annealed sense: there exists a positive definite symmetric matrix  $\Sigma(\beta)$ such that for any $f\in C_b(\R^d)$, we have
 \begin{equation}\label{e.cltwinding}
\EE \int_{\R^d} f(\frac{x}{\sqrt{t}}) \sfp(t, x)dx\to \int_{\R^d} f(x)\Phi(x)dx,
\end{equation}
as $t\to\infty$. Here
$$
\Phi(x)=\frac{1}{(2\pi)^{d/2}\sqrt{|\Sigma(\beta)|}}\exp\left\{-\frac12x^T\cdot
\Sigma^{-1}(\beta)x\right\}
$$
is the density of a $d$-dimensional centered Gaussian vector whose
covariance matrix equals $\Sigma(\beta)$. 

In the  case of $1+1$ spacetime white noise, $\Sigma(\beta)$ admits the following representation
\begin{equation}\label{e.exvarwinding}
\Sigma(\beta)=1+\beta^2\E_{W_1}\E_{W_3}\big[\mathcal{A}(\beta,W_1,W_3)^2\big],
\end{equation}
where
\[
 \mathcal{A}(\beta,W_1,W_3)
    =\int_{\bT^2} \Xi(\beta,y,W_1)
    \left(\frac{e^{\beta W_1(z)+\beta W_3(z)}}{\int_{\bT}
    e^{\beta W_1(z')+\beta
      W_3(z')}dz'}-1\right)1_{[0,y]}(z)dydz,
      \]
      with
      \[
\Xi(\beta,y,W_1)=\E_{W_2}\left[\frac{ e^{\beta W_2(y) -\beta
    W_1(y)}}{\Big(\int_{\bT}e^{\beta W_2(y') -\beta W_1(y')}dy'\Big)^2}\right].
\]
Here $\{W_i\}_{i=1,2,3}$ are   independent copies of standard Brownian
bridges connecting $(0,0)$ with $(1,0)$ and $\E_{W_i}$ is the
expectation performed only on $W_i$.
\end{theorem}

To get a sense why the central limit theorem holds for the endpoint displacement in the periodic setting, we consider  the formal partition function of the polymer measure
\[
\E_B \exp\left\{\beta\int_0^N\xi(s,B_s)ds\right\}.
\]
Here, to simplify we have assumed that $t=N$ is a positive integer. Since $\xi$ is spatially periodic, one can view $B$ either as a path on $\R^d$ or $\bT^d$. For our purpose, because we are interested in the endpoint displacement, we choose $\R^d$. For each integer $k\geq0 $, let $w_k=B_k-\lfloor B_k\rfloor$ so that it takes values on $\bT^d$ and records the position of the polymer path on the cylinder  when viewed from  the second perspective. Due to the spatial periodicity,    the random variables $\{w_k\}$ mix exponentially fast, as a consequence of  Theorem~\ref{t.ofos}. Additionally, since $\xi$ is white in time, conditioning on the positions of all $w_k$, the winding numbers accumulated in each time interval $[k-1,k]$, denoted by $\eta_k$, are independent. Therefore, to study the endpoint displacement, or equivalently, the total winding number $\sum_{k=1}^N \eta_k$, the correlation among the summands only comes from those of $\{w_k\}$, and we have the sum of a sequence of weakly dependent random variables, which is expected to satisfy the central limit theorem.

The above heuristics are likely the most straightforward to verify
rigorously, and we will provide more details in Section~\ref{s.sumiid}
below. Ultimately, the endpoint displacement is approximated by the
sum of a stationary sequence of weakly dependent random variables,
each representing the winding number of the polymer path accumulated
in a  sub-interval. To conclude the central limit theorem,  it is
enough to check the standard $\rho-$mixing condition \cite[Section 19]{bil}. In this case, the limiting variance    is naturally expressed in terms of the sum of covariance, see \eqref{e.sumco} below. A similar argument can be found in a last passage percolation model \cite{bakhtin1}.

Since the polymer endpoint is expected to be super-diffusive in the
non-periodic setting (see \cite[Theorem 1.11]{CH16} for the case of
$1+1$ spacetime white noise), it is natural to ask if we can study the
asymptotics of $\Sigma(\beta)$ as $L$ - the torus length - tends to $\infty$. The effective diffusivity presumably diverges, with the diverging speed hinting on  the super-diffusive exponent. From the sum of covariance formula for $\Sigma(\beta)$, it seems unlikely that any proper bound on $\Sigma(\beta)$ could be obtained. This motivated us to develop a different proof of the central limit theorem, leading to the explicit diffusivity in \eqref{e.exvarwinding} for the case of $1+1$ spacetime white noise. Unlike \eqref{e.exvar}, the formula $\Sigma(\beta)$ does not extend to the smooth noise case, and we will explain later why our proof is restricted to the white noise.  The proof again relies on the Clark-Ocone formula, as in Section~\ref{s.clarkocone}, but in this case, the calculation of the Malliavin derivative and the subsequent analysis is significantly more complicated. We present the main ideas of the proof in Section~\ref{s.timereversal} below.

In the following two sections, we present the proof of \eqref{e.cltwinding} and \eqref{e.exvarwinding} respectively.

\subsection{Sum of weakly dependent random variables}
\label{s.sumiid}

In this section, we present a proof of \eqref{e.cltwinding} by a
standard argument used for sum of weakly dependent random
variables.  The argument relies on a Markov chain representation of the winding number and the fast mixing of the projective process established in Theorem~\ref{t.ofos}.

Recall that we are interested in studying the quenched density $\sfp$ defined in \eqref{e.defsfp}, which is proportional to $\sfZ$. Our first step is to rewrite $\sfZ$ by constructing a Markov chain representing the winding number of the polymer path accumulated in each sub-interval. The construction is standard and can be found, for example, in \cite{Sin91}. However, it was very helpful for us to  see the underlying structure, so we present it here.

\subsubsection{Markov chain for winding number}
We need another notation: let $\cZ$ be the Green's function of SHE on $\R_+\times\R^d$, 
\begin{equation}\label{e.defGreenZ}
\begin{aligned}
&\partial_t \cZ_{t,s}(x,y)=\frac12\Delta_x \cZ_{t,s}(x,y)+\beta \cZ_{t,s}(x,y)\xi(t,x), \quad\quad t>s,x\in\R^d,\\
&\cZ_{s,s}(x,y)=\delta(x-y), 
\end{aligned}
\end{equation}
thus, the periodic Green's function is 
\[
\G_{t,s}(x,y)=\sum_{j \in\Z^d} \cZ_{t,s}(x,y+j)=\sum_{j \in\Z^d} \cZ_{t,s}(x+j,y).
\]
With the above notation, we have $\sfZ(t,x)=\cZ_{t,0}(x,0)$. 

%

Fix any $t>0$ and denote $N=\lf t\rf$. Fix $y\in\R^d, y'\in\bT^d$. We
first write $y$ as $y=j_{N+1}+x_{N+1}$ for some $j_{N+1}\in \Z^d$ and
$x_{N+1}\in[0,1)^d$. Since $\cZ$ is the Green's function  of SHE, we have 
\[
\begin{aligned}
\cZ_{t,0}(y,y')&=\int_{\R^d}\cZ_{t,N}(j_{N+1}+x_{N+1},x)\cZ_{N,0}(x,y')dx\\
&=\sum_{j_N\in\Z^d}\int_{\bT^d} \cZ_{t,N}(j_{N+1}+x_{N+1},j_N+x_N)\cZ_{N,0}(j_N+x_N,y')dx_N.
\end{aligned}
\]
Iterating the above relation, we obtain
\[
\begin{aligned}
\cZ_{t,0}(y,y')= \sum_{j_1,\ldots,j_N\in\Z^d}&\int_{\bT^{Nd}}\cZ_{t,N}(j_{N+1}+x_{N+1},j_N+x_N) \\
&\times \prod_{k=1}^N \cZ_{k,k-1}(j_k+x_k,j_{k-1}+x_{k-1})d\x_{1,N},
\end{aligned}
\]
where we used the simplified notation $d\x_{1,N}=dx_1\ldots dx_{N}$ and the convention 
$
j_0=0,x_0=y'.
$
 In other words,  we have decomposed   $\R^d=\cup_{j\in\Z^d} \{j+[0,1)^d\}$, then integrated in each interval and summed  them up. One should think of the variable $j_k+x_k$ as representing the location of the polymer path at time $k$, with $j_k$ the integer part and $x_k$ the fractional part.

Now we make use of the periodicity to note that 
\begin{equation}\label{e.peZG}
\begin{aligned}
&\sum_{j_k\in\Z^d}\cZ_{k,k-1}(j_k+x_k,j_{k-1}+x_{k-1})=\sum_{j_k\in\Z^d} \cZ_{k,k-1}(j_k-j_{k-1}+x_k,x_{k-1})\\
&=\sum_{j_k\in\Z^d} \cZ_{k,k-1}(j_k+x_k,x_{k-1})=\G_{k,k-1}(x_k,x_{k-1}).
\end{aligned}
\end{equation}
 Thus $\cZ_{t,0}$ can be rewritten as 
\begin{equation}\label{e.7151}
\begin{aligned}
\cZ_{t,0}(y,y')=\int_{\bT^{Nd}}  &\left(\sum_{j_1,\ldots,j_{N}\in\Z^d} \frac{\cZ_{t,N}(j_{N+1}+x_{N+1},j_N+x_N)\prod_{k=1}^N \cZ_{k,k-1}(j_k+x_k,j_{k-1}+x_{k-1})  }{\G_{t,N}(x_{N+1},x_N)\prod_{k=1}^N \G_{k,k-1}(x_k,x_{k-1}) }\right)\\
&\times \G_{t,N}(x_{N+1},x_N)\prod_{k=1}^N \G_{k,k-1}(x_k,x_{k-1})  dx_{1,N}.
\end{aligned}
\end{equation}

If we further sum over $j_{N+1}$, the term inside the parentheses in \eqref{e.7151} equals to $1$, and this inspires us to define a Markov chain as follows. Fix the realization of $\xi$ and  the vector
\begin{equation}\label{e.bx}
\x:=(x_0,\ldots,x_{N+1})\in \bT^{(N+2)d}.
\end{equation} 
We construct an integer-valued, time inhomogeneous Markov chain $\{Y_j\}_{j=1}^N$, with 
\begin{equation}\label{e.6174}
\begin{aligned}
& \Pb_\x[Y_1=j_1]= \frac{\cZ_{1,0}(j_1+x_1,x_0)}{\G_{1,0}(x_1,x_0)},\\
&\Pb_\x[Y_2=j_2|Y_1=j_1]= \frac{\cZ_{2,1}(j_2+x_2,j_1+x_1)}{\G_{2,1}(x_2,x_1)},\\
&\ldots\\
&\Pb_\x[Y_N=j_N|Y_{N-1}=j_{N-1}]=\frac{\cZ_{N,N-1}(j_N+x_N,j_{N-1}+x_{N-1})}{\G_{N,N-1}(x_N,x_{N-1})},\\
&\Pb_\x[Y_{N+1}=j_{N+1}|Y_{N}=j_{N}]=\frac{\cZ_{t,N}(j_{N+1}+x_{N+1},j_{N}+x_{N})}{\G_{t,N}(x_{N+1},x_{N})}.
\end{aligned}
\end{equation}
By \eqref{e.peZG} and \eqref{e.6174}, it is clear that $Y_{N+1}$ is a sum of independent random variables, for each fixed realization of the noise and $\x$. We rewrite it as 
\begin{equation}\label{e.defeta}
Y_{N+1}=\sum_{k=1}^{N+1} \eta_k, \quad\quad\mbox{ with } Y_0=0, \quad   \eta_k=Y_k-Y_{k-1},
\end{equation}
and one can interpret $\eta_k$ as the winding number accumulated during the time interval $[k-1,k]$ for $k=1,\ldots,N$, and $\eta_{N+1}$ corresponds to the winding number in $[N,t]$. 
We have 
\begin{equation}\label{e.laweta}
\begin{aligned}
&\Pb_\x[\eta_k=j]=\frac{\cZ_{k,k-1}(x_k+j,x_{k-1})}{\G_{k,k-1}(x_k,x_{k-1})}, 
\quad\quad j\in \Z^d, \quad k=1,\ldots,N,\\
&\Pb_\x[\eta_{N+1}=j]=\frac{\cZ_{t,N}(x_{N+1}+j,x_{N})}{\G_{t,N}(x_{N+1},x_{N})}, 
\quad\quad j\in \Z^d.
\end{aligned}
\end{equation}

With the Markov chain,  the summation in \eqref{e.7151} takes the simple form 
\[
\begin{aligned}
&\sum_{j_1,\ldots,j_{N}\in \Z^d} \frac{\cZ_{t,N}(j_{N+1}+x_{N+1},j_N+x_N)\prod_{k=1}^N \cZ_{k,k-1}(j_k+x_k,j_{k-1}+x_{k-1})  }{\G_{t,N}(x_{N+1},x_N)\prod_{k=1}^N \G_{k,k-1}(x_k,x_{k-1}) }\\
&=\Pb_{\mathbf{x}}[Y_{N+1}=j_{N+1}].
\end{aligned}
\]

Therefore, \eqref{e.7151} is rewritten as 
\begin{equation}\label{e.6172}
\begin{aligned}
\cZ_{t,0}(y,y')=\int_{\bT^{Nd}}\Pb_{\mathbf{x}}[Y_{N+1}=j_{N+1}] \G_{t,N}(x_{N+1},x_N)\prod_{k=1}^N \G_{k,k-1}(x_k,x_{k-1})    d\x_{1,N},
\end{aligned}
\end{equation}
and the quenched density is 
\begin{equation}\label{e.7301}
\begin{aligned}
\sfp(t,y)&=\frac{\cZ_{t,0}(y,0)}{\int_{\R^d} \cZ_{t,0}(y',0)dy'}\\
&=\frac{\int_{\bT^{Nd}}\Pb_{\mathbf{x}}[Y_{N+1}=j_{N+1}] \G_{t,N}(x_{N+1},x_N)\prod_{k=1}^N \G_{k,k-1}(x_k,x_{k-1})    d\x_{1,N}}{\int_{\bT^{(N+1)d}}\G_{t,N}(x_{N+1},x_N)\prod_{k=1}^N \G_{k,k-1}(x_k,x_{k-1})    d\x_{1,N+1}},
\end{aligned}
\end{equation}
with $x_0=0$ and $y=j_{N+1}+x_{N+1}$. 

In other words, to sample  from $\sfp(t,\cdot)$, there are two steps:

(i)  sample $w_1,\ldots,w_{N+1}$ from a density (on $\bT^{(N+1)d}$) proportional to 
\[
\G_{t,N}(x_{N+1},x_N)\prod_{k=1}^N \G_{k,k-1}(x_k,x_{k-1});
\]

(ii) given the values of $w_1=x_1,\ldots,w_{N+1}=x_{N+1}$, we run the Markov chain defined in \eqref{e.6174} and obtain a sample of $Y_{N+1}$. 

(iii) the random variable $Y_{N+1}+w_{N+1}$ is distributed with the
density $\sfp(t,\cdot)$.  

Our goal is to show a central limit theorem for
$\frac{Y_{N+1}+w_{N+1}}{\sqrt{t}}$ as $t\to\infty$
($N=\floor{t}$). By the fact that $w_{N+1}\in [0,1)^d$,  it suffices to consider $\frac{Y_{N+1}}{\sqrt{t}}$, so we can actually integrate out the variable $x_{N+1}$ in \eqref{e.7301}. 

Let us introduce some more notation.    For any $\nu_1,\nu_2\in \mathcal{M}_1(\bT^d)$, define $\mu_t(d\x;\nu_1,\nu_2)$ as the probability on $\bT^{(N+2)d}$ such that 
\[
\mu_t(d\x;\nu_1,\nu_2)=\frac{\nu_1(dx_{N+1})\G_{t,N}(x_{N+1},x_N)\prod_{k=1}^N \G_{k,k-1}(x_k,x_{k-1})\nu_2(dx_0)d\x_{1,N}}{C_t(\nu_1,\nu_2)},
\]
with $C_t(\nu_1,\nu_2)$   the normalization constant. One should view $\mu_t$ as the quenched density of the polymer path on the cylinder at integer time points, and the distributions at the boundaries are given by $\nu_1,\nu_2$ respectively.

In this way, the random variable $Y_{N+1}$ obtained through (i)-(ii) has the distribution given by (we use $\mathsf{P}$ to denote the quenched probability)
\begin{equation}\label{e.qpwinding}
\mathsf{P}[Y_{N+1}=j_{N+1}]=\int_{\bT^{(N+2)d}}\Pb_{\mathbf{x}}[Y_{N+1}=j_{N+1}]\mu_t(d\x;\mathrm{m},\delta_0),
\end{equation}
where $\mathrm{m}$ is the Lebesgue measure on $\bT^d$ and $\delta_0$ is the Dirac mass on $\bT^d$ at the origin, which is the starting point of our analysis.


\subsubsection{Approximation and $\rho-$mixing}

To make use of the representation of the quenched probability in \eqref{e.qpwinding}, the key is to realize that  the two distributions $\mathrm{m},\delta_0$ at the boundaries  play no role in large time, which turns out to be a consequence of Theorem~\ref{t.ofos}. 

As a matter of fact, with any distributions $\nu_1,\nu_2$ imposed at the boundaries, suppose we are interested in the marginal distribution of $x_k$ with some $k$ in the middle, then  after integrating out all other variables, we obtain the marginal density of $x_k$ as
\[
\begin{aligned}
&\int_{\bT^{(N+1)d}} \frac{\nu_1(dx_{N+1})\G_{t,N}(x_{N+1},x_N)\prod_{k=1}^N \G_{k,k-1}(x_k,x_{k-1})\nu_2(dx_0)d\x_{1,k-1}d\x_{k+1,N}}{C_t(\nu_1,\nu_2)}\\
&=\frac{\rhob(t,\nu_1;k,x_k)\rhof(k,x_k;0,\nu_2)}{\int_{\bT^d}\rhob(t,\nu_1;k,x_k')\rhof(k,x_k';0,\nu_2) dx_k'},
\end{aligned}
\]
with $\rhob,\rhof$ defined in \eqref{e.forwardbackward}. If $t-k\gg1,
k\gg1$, then by \eqref{051705-23}, one can replace $\nu_1,\nu_2$ above
by arbitrary distributions with an exponentially small error. We are
interested in the total winding number
$Y_{N+1}=\sum_{k=1}^{N+1}\eta_k$, with the distribution of $\eta_k$
depending on the choice of $(x_{k-1},x_k)$, thus, different choices of
$\nu_1,\nu_2$ only affect the distributions of $\eta_k$ near the
boundaries, i.e., those $k$ either close to $1$, or
  $N$. To study $\frac{Y_{N+1}}{\sqrt{N+1}}$, those
$\eta_k$-s   should not contribute, as $N\to\infty$. In particular, we
can replace $\mu_t(d\x;\mathrm{m},\delta_0)$ by
$\mu_t(d\x;\varrho_1,\varrho_2)$ in \eqref{e.qpwinding}, with
$\varrho_1,\varrho_2$ sampled independently from the invariant measure
of the projective process, so that the polymer path on the cylinder
is ``at stationarity''. The induced error in the replacement can be
estimated using \eqref{051705-23} in a rather straightforward way, see
\cite[Section 3.2]{YGTK22} for details.

Now suppose $t=N+1$, we study \eqref{e.qpwinding} with $\mu_t(d\x;\mathrm{m},\delta_0)$ replaced by $\mu_t(d\x;\varrho_1,\varrho_2)$. Since the two endpoints of the directed polymer on the cylinder is at stationarity, one can  construct  
 a path measure to realize $\{\eta_k\}_k$ as a sequence
of stationary random variables, see \cite[Section 3.3]{YGTK22}. To
show the central limit theorem of
$\frac{1}{\sqrt{N+1}}\sum_{k=1}^{N+1}\eta_k$,  it is enough to check
whether $\{\eta_k\}_k$  satisfies a
suitable mixing condition. 

Again, the mixing of $\{\eta_k\}_k$ comes from the respective
  property  of the projective process.   To
illustrate the ideas, we consider $\eta_i$ and $\eta_j$, with
$i<j$. The correlation between $\eta_i$ and $\eta_j$ only comes from
the correlation between $(w_{i-1},w_i)$ and $(w_{j-1},w_j)$,
   which are sampled from the joint density
$\mu_t(d\x;\varrho_1,\varrho_2)$ (in step (i)). To see that the correlation is small
for $j-i\gg1$,  we consider $w_i,w_{j-1}$  as
an example. Integrating all other variables, the joint density of $w_i,w_{j-1}$ is proportional to 
\[
\rhob(t,\varrho_1;j-1,x_{j-1})\G_{j-1,i}(x_{j-1},x_i)\rhof(i,x_i;0,\varrho_2).
\]
Since
$\rhob(t,\varrho_1;j-1,x_{j-1})$ and $\rhof(i,x_i;0,\varrho_2)$ are independent, the factor $\G_{j-1,i}(x_{j-1},x_i)$ encodes  the correlation, and one may expect that, with $j-i\gg1$, it approximately factorizes into the product of  functions of $x_{j-1}$ and $x_i$.  Indeed, take the midpoint $k=\lf \frac{i+j}{2}\rf$, then the above expression can be rewritten as 
\[
\begin{aligned}
&\int_{\bT^d}\rhob(t,\varrho_1;j-1,x_{j-1})\G_{j-1,k}(x_{j-1},x_k)\G_{k,i}(x_k,x_i)\rhof(i,x_i;0,\varrho_2)dx_k\\
&=\int_{\bT^d}\rhob(t,\varrho_1;j-1,x_{j-1})\rhof(j-1,x_{j-1};k,x_k)\rhob(k,x_k;i,x_i)\rhof(i,x_i;0,\varrho_2) F(x_k)dx_k.
\end{aligned}
\]
Here $F$ is some function we do not need to write explicitly. With  $j-1-k\gg1$ and $k-i\gg1$, applying  \eqref{051705-23} again, one can replace $\rhof(j-1,x_{j-1};k,x_k), \rhob(k,x_k;i,x_i)$ by $\rhof(j-1,x_{j-1};k,\varrho_3), \rhob(k,\varrho_4;i,x_i)$ respectively, with an exponentially small error, where $\varrho_3,\varrho_4$ are sampled independently from the invariant measure. Thus, up to a small error, the joint density of $x_i,x_{j-1}$ is proportional to 
\[
\rhob(t,\varrho_1;j-1,x_{j-1})\rhof(j-1,x_{j-1};k,\varrho_3)\rhob(k,\varrho_4;i,x_i)\rhof(i,x_i;0,\varrho_2).
\]
Up to a normalization, this is precisely the product of the marginal densities of $w_i$ and $w_{j-1}$. Thus, the above discussion suggests that, for $j-i\gg1$, the joint density of $x_i,x_{j-1}$ can be approximated by the product of the marginal densities.

The heuristic argument was made precise in \cite[Section
3.4]{YGTK22}. In particular, the $\rho-$mixing coefficient was shown
to decay exponentially fast \cite[Proposition 3.7]{YGTK22}. By a
standard argument based on \cite[Theorem 19.2]{bil}, this implies the central limit theorem for $\frac{1}{\sqrt{N+1}}Y_{N+1}=\frac{1}{\sqrt{N+1}}\sum_{k=1}^{N+1}\eta_k$, with the limiting covariance matrix taking the form 
\begin{equation}\label{e.sumco}
\Sigma(\beta)= \sum_{k\in\Z} \EE \Big[\,\eta_0\otimes \eta_k\Big].
\end{equation}

For the central limit theorem of sum of stationary random variables satisfying the $\rho-$mixing condition, it is usually a separate issue to verify that the limiting covariance matrix is strictly positive definite. In our case, it comes from the fact that
\begin{equation}\label{e.positivedefinite}
\EE \left(\int_{\R^d}x_ix_j \sfp(t,x)dx-\int_{\R^d} x_i\sfp(t,x)dx\int_{\R^d} x_j\sfp(t,x)dx\right) =\delta_{ij}t.
\end{equation}
It results from the  shear invariance  
of the noise, namely the fact that
$\{\xi(t,x)\}_{t,x}\stackrel{\text{law}}{=}\{\xi(t,x+\theta
t)\}_{t,x}$ for any $\theta\in\R^d$. As a result  
\begin{equation}
\EE \log \int_{\R^d} e^{\theta \cdot x} \sfp(t,x)dx=\frac12|\theta|^2 t+\EE \log \int_{\R^d}  \sfp(t,x)dx,\quad\quad\theta\in\R^d.
\end{equation}
 Taking the second
derivative in $\theta$ at $\theta=0$ we conclude \eqref{e.positivedefinite}.

From \eqref{e.positivedefinite}, we derive that for any $\theta\in\R^d$, 
\[
\EE \int_{\R^d}|\theta \cdot x|^2 \sfp(t,x)dx \geq \EE \int_{\R^d}|\theta \cdot x|^2 \sfp(t,x)dx-\EE (\int_{\R^d} \theta\cdot x \sfp(t,x)dx)^2  \geq |\theta|^2t,
\]
which implies that $\Sigma(\beta)\geq I_d$, where $I_d$ is the
$d\times d$ identity matrix. For a complete proof of this result in
$d=1$, we refer to \cite[Section 5]{YGTK22}. The multidimensional case is treated in the same way.

\subsection{Formula for the  diffusivity}
\label{s.timereversal}
 
The  formula \eqref{e.sumco} is hard to analyze quantitatively, much like the homogenization constant in Section~\ref{s.homogenization}, which is expressed in terms of the corrector that solves an abstract Poisson equation \eqref{e.sigmaex1}. In both cases, the effective diffusivities depend on the temporal correlation, a well-known fact. The dependence is explicit in \eqref{e.sumco} since $k$ is the time variable, while in \eqref{e.sigmaex1}, it is the gradient of the corrector $\mathcal{D}\chi$ that encodes the temporal correlation.

For the $1+1$ spacetime white noise, the effective diffusivity
$\Sigma(\beta)$    admits an expression only in terms
of the underlying invariant measure \eqref{e.exvarwinding}. The goal of this section is to  sketch the proof of this result.

To study the total variance, one can write it as 
\[
\begin{aligned}
\EE \int_{\R} x^2 \sfp(t,x)dx=&\EE \left(\int_{\R} x\sfp(t,x)dx\right)^2+ \EE \left(\int_{\R}x^2\sfp(t,x)dx-(\int_{\R} x\sfp(t,x)dx)^2 \right).
\end{aligned}
\]
The first term on the r.h.s. is the variance of the quenched mean,
while the second term is the mean of the quenched variance, which
equals to $t$ by \eqref{e.positivedefinite}. Therefore, the problem
reduces to studying the varianceof the quenched average of the
endpoint   of the polymer 
\[
X_t:=\int_{\R} x\sfp(t,x)dx.
\]
 We have  \begin{equation}
  \label{011308-24}
\lim_{t\to\infty}\frac{1}{t}\Var[ X_t]= \Sigma(\beta)-1,
\end{equation}
with $\Sigma(\beta)$ given by \eqref{e.exvarwinding}.

\subsubsection{It\^o integral representation}
As in the study of the height function $h$, one may want to apply  the Clark-Ocone formula to rewrite the quenched mean as (note $\EE X_t=0$ by symmetry)
\[
X_t=\int_0^t\int_{\bT} \EE[D_{s,y}X_t|\F_s]\xi(s,y)dyds.
\]
Here we view $X_t$ as a functional of the noise $\xi$ on $\R_+\times
\bT$, and $D_{s,y}$ is the Malliavin derivative associated with
it. Suppose we can compute $\EE[D_{s,y}X_t|\F_s]$ and understand its
asymptotic behavior for $s,t-s\gg1$, then
there is a chance to obtain a new expression for the left hand side of
\eqref{011308-24}, more explicit than \eqref{e.sumco}.

We might attempt to directly compute $D_{s,y}X_t=\int_{\R}x D_{s,y}\sfp(t,x)dx$, but analyzing the conditional expectation $\EE[D_{s,y}X_t|\F_s]$ does not appear straightforward. Instead, we employ the following trick: define the tilted partition function
\[
\mathcal{E}_\theta(t)=\int_{\R} e^{\theta x} \sfZ(t,x)dx,\quad\quad \theta\in\R,
\]
so the quenched mean can be expressed  as $X_t=\partial_\theta\log \mathcal{E}_\theta(t)\big|_{\theta=0}$. 

One can view $\log \mathcal{E}_\theta(t)$  as a tilted height function, and the idea is to first write down the Clark-Ocone representation of $\log  \mathcal{E}_\theta(t)$, which shares some similarity with the one for $h(t,x)$ in Section~\ref{s.clarkocone}, except for the difference arising from the tilt, then we take the derivative with respect to $\theta$ on both sides to obtain the It\^o integral representation for $X_t$.

It turns out to be a complicated calculation. For the details, we refer to \cite[Eq. (2.18), Lemma 2.3]{YGTK23}. The It\^o integral representation of $X_t$ takes  the form 
\begin{equation}\label{e.coXt}
X_t=\beta \int_0^t\int_{\bT} \EE[I_t(s,y)|\F_s] \xi(s,y)dyds,
\end{equation}
with 
\begin{equation}\label{e.i1minusi2}
I_t(s,y)=\rho_{\rm m}(t,-|s,y)\int_{\bT}[\psi(s,y')-\psi(s,y)]\rho_{\rm m}(t,-|s,y')dy',
\end{equation}
where $\rho_{\rm m}$ is the  midpoint density (at time $s$)  of the
point-to-line directed polymer of length $t$  on the cylinder: with $\G_{t,s}(-,y):=\int_{\bT} \G_{t,s}(x,y)dx$, we have
%
\begin{equation}\label{e.defrho}
\rho_{\rm m}(t,-|s,y)=\frac{ \G_{t,s}(-,y) \G_{s,0}(y,0)}{\G_{t,0}(-,0)},
\end{equation}
and $\psi$ can be interpreted as the quenched average of the winding number    of the point-to-point polymer path on the cylinder, starting at $(s,y)$ and ending at $(0,0)$: 
\begin{equation}\label{e.defW}
\psi(s,y)=\frac{\sum_{n\in\Z} (n-y)\cZ_{s,0}(y,n)}{\G_{s,0}(y,0)}.
\end{equation}
Alternatively, if we view the polymer path as lying on the plane, then it starts from $(s,y)$ and runs backwards in time, with a terminal condition at time zero so that it is pinned at integer points, see \cite[Lemma 3.1]{YGTK23}.

For the height function, the Malliavin derivative is given by formula
\eqref{e.madeh},
from which one can directly derive its limit as both $s$ and
$t-s\to\infty$ using Theorem~\ref{t.ofos}. The expression of
$I_t(s,y)$ is more complicated, and it is unclear what should be the
limit of $I_{t}(s,y)$ as $s,t-s\to\infty$. There, two types of factors appear:   
the midpoint density  and the quenched average of the
winding number.  The midpoint density   can be dealt with using
Theorem~\ref{t.ofos} and we obtain that for $s \gg 1$ and $t-s\gg1$,  
\[
\rho_{\rm m}(t,-|s,y)\approx \frac{\rhob(t,\varrho_1;s,y)\rhof(s,y;0,\varrho_2)}{\int_{\bT}\rhob(t,\varrho_1;s,y')\rhof(s,y';0,\varrho_2)dy'}\stackrel{\text{law}}{=}\frac{e^{\beta W_1(y)+\beta W_2(y)}}{\int_0^1 e^{\beta W_1(y')+\beta W_2(y')}dy'},
\]
where $\varrho_1,\varrho_2$ are independently sampled from the
invariant measure of the projective process, and $W_1,W_2$ are two
independent standard Brownian bridges connecting $(0,0)$ and
$(1,0)$. The key difficulty lies in the study of
$\psi(s,y')-\psi(s,y)$, in particular, we need to determine the joint
distribution of $\{\psi(s,y')-\psi(s,y)\}_{y,y'\in\bT}$ and
$\{\rho_{\rm m}(t,-|s,y)\}_{y\in\bT}$ for $ s \gg 1$ and $t-s\gg1$.
%
%

\subsubsection{Diffusion in Burgers drift and time reversal}

To study $\psi(s,y')-\psi(s,y)$, or equivalently $\nabla \psi(s,y)$,
we first observe that $\psi(s,y)$ can be interpreted as the quenched
mean of the winding number of a polymer path on the cylinder,
starting from $(s,y)$, running backwards in time and ending at
$(0,0)$. Thus, one needs to examine the difference between the
quenched means of two polymer paths starting from different points. A
novelty of our argument is drawing the connection between the polymer
measure and a diffusion in a random environment, which allows the
quenched mean to be expressed in terms of the solution to the respective backward Kolmogorov equation. 

We first state a general fact which relates the polymer measure to a diffusion in a random environment. Consider the polymer path starting from $(t,x)$, running backwards in time, with a terminal condition $Z_0$: the formal partition function of the Gibbs measure takes the form
\[
Z(t,x)=\E_B e^{\beta\int_0^t \xi(t-s,x+B_s)ds}Z_0(x+B_t).
\]
One can show that $\{x+B_s\}_{s\in[0,t]}$ under the polymer measure has the same law as the   diffusion $\{\mathcal{X}_s\}_{s\in[0,t]}$ solving the following SDE
\begin{equation}\label{e.diff1}
\begin{aligned}
&d\mathcal{X}_s=u(t-s,\mathcal{X}_s)ds+dB_s,\quad\quad \mathcal{X}_0=x,\\
&u(s,y)=\nabla \log Z(s,y).
\end{aligned}
\end{equation}
%
%
In the smooth noise case, this fact can be proved by the Girsanov
theorem \cite[Lemma 4.2]{YGTK23}; in the white noise case, the drift
$u$ is distribution-valued, so it is a singular diffusion, which can
be made sense of by tools from rough path and singular SPDEs
\cite{CC18,DD16}. 

With the diffusion representation, the quenched mean of the polymer path can be represented by the solution to the  following backward Kolmogorov equation
\begin{equation}\label{e.bapde}
\begin{aligned}
&\partial_t \phi=\frac12\Delta\phi+u\nabla \phi+u,\\
&\phi(0,x)=0.
\end{aligned}
\end{equation}
As a matter of fact,  we have 
\[
\E_B \mathcal{X}_t=x+\phi(t,x).
\]
The above identity comes from a probabilistic representation: from
\eqref{e.bapde}, we see  the solution $\phi(t,x)$ equals to the
forcing $u$ integrated along the  random   characteristics, which is precisely the trajectory of $\mathcal{X}_s$, and this leads to the quenched mean. For the proof in the case of a smooth noise, see \cite[Lemma 4.3]{YGTK23}.

Here lies the crux of the argument.  Since $\psi$ itself is the quenched mean of the displacement, we have
\[
\nabla\psi(s,y)= \nabla\phi(s,y),
\]
with a properly chosen $Z_0$. Starting from the backward equation \eqref{e.bapde}, we derive that $g:=1+\nabla\phi$ satisfies a forward equation
\begin{equation}\label{e.fopde}
\begin{aligned}
&\partial_tg=\frac12\Delta g+\nabla (ug),\\
&g(0,x)=1.
\end{aligned}
\end{equation}
This is the Fokker-Planck equation for the diffusion
\begin{equation}\label{e.diff2}
d\mathcal{Y}_s=-u(s,\mathcal{Y}_s)ds+dB_s,    \quad\quad \cY_0\sim \mathrm{m}.
\end{equation}
The symbol  $X\sim\mu$ means that a random variable $X$ is distributed
according to $\mu$. Here $\mathrm{m}$ is the normalized Lebesgue
measure on $\bT$. In other
words,  studying $\nabla\psi$ is equivalent with considering $g$, which is the quenched density of $\mathcal{Y}$. On the other hand, from the expression of $\rho_{\rm m}$ in \eqref{e.defrho}, we can rewrite it as 
\[
\begin{aligned}
\rho_{\rm m}(t,-|s,y)=\frac{ \G_{t,s}(-,y) \G_{s,0}(y,0)}{\G_{t,0}(-,0)}= \frac{ \G_{t,s}(-,y) e^{\int_0^y   \tilde{u}(s,z)dz}}{\int_{\bT} \G_{t,s}(-,y')e^{\int_0^{y'}  \tilde{u}(s,z)dz}dy'},
\end{aligned}
\]
with $\tilde{u}(s,y)=\nabla \log \G_{s,0}(y,0)$. Since the goal was to understand the joint distribution of $\{\rho_{\rm m}(t,-|s,y)\}_{y\in\bT}$ and $\{\nabla\psi(s,y)=g(s,y)-1\}_{y\in\bT}$, and  the factor $\G_{t,s}(-,\cdot)$ appearing in the expression of $\rho_{\rm m}$  is independent of $\F_s$, it reduces to studying the joint distribution of $\tilde{u}(s,\cdot)$ and $g(s,\cdot)$. However, $\tilde{u}$ solves the same stochastic Burgers equation as  $u=\nabla\log Z=\nabla h$, with the only difference coming from the initial data. In the periodic setting, similar to the projective process, there is a one-force-one-solution principle for the stochastic Burgers equation \cite{Ros21}, and something similar to \eqref{051705-23} holds, which implies that $\tilde{u}(s,\cdot)\approx u(s,\cdot)$ for $s\gg1$. Therefore, the problem reduces to studying the joint distribution of $g(s,\cdot)$ and $u(s,\cdot)$, the solution to the Fokker-Planck equation and its coefficient.

In general, it is impossible to derive the joint distribution of the solution to a Fokker-Planck equation and its coefficient.  However, for the $1+1$ white noise, there is an amazing time-reversal skew-symmetry for the solution to the stochastic Burgers equation at stationarity: suppose $Z_0(x)=e^{\beta W(x)}$ for some standard Brownian bridge $W$, then $u=\nabla \log Z$ satisfies 
\begin{equation}\label{e.burgersAsym}
\{u(s,\cdot)\}_{s\in[0,t]}\stackrel{\text{law}}{=}\{-u(t-s,\cdot)\}_{s\in[0,t]}.
\end{equation}
The above identity in law  is the only place where we restrict the argument to the white noise -- for the smooth noise, we do not expect it to hold.

With \eqref{e.burgersAsym}, we consider a third diffusion process solving the SDE 
\begin{equation}\label{e.diff3}
d\tilde{\cY}_s=u(t-s,\tilde{\cY}_s)ds+dB_s, \quad\quad \tilde{\cY}_0\sim \mathrm{m}.
\end{equation}
Compare to \eqref{e.diff2}, $\cY$ solves the same SDE with a time-reversed drift of the opposite sign. Let $\tilde{g}$ be the quenched density of $\tilde{\cY}$, then the identity in law \eqref{e.burgersAsym} implies that, for any $t>0$, 
\begin{equation}\label{e.jointgu}
(g(t,\cdot),u(t,\cdot))\stackrel{\text{law}}{=}(\tilde{g}(t,\cdot),-u(0,\cdot)).
\end{equation}
This is the key identity used in \cite{YGTK23}. Now we compare
\eqref{e.diff3} with \eqref{e.diff1}, the two diffusions share the
same drift, with the only difference coming from the initial
condition. Therefore, one can go back to the polymer description and
view $\tilde{\cY}$ as a polymer path on a cylinder with appropriate
boundary conditions, and we actually have
 \begin{align*}
\tilde{g}(t,x)=\int_{\bT}\left( \frac{  \G_{t,0}(y,x) e^{\beta
                 W(x)}}{\int_{\bT} \G_{t,0}(y,x')e^{\beta
                 W(x')}dx'}\right) dy
    =\int_{\bT}\frac{\rhob(t,y;0,x)e^{\beta W(x)}}{\int_{\bT}\rhob(t,y;0,x')e^{\beta W(x')}dx'}dy.
\end{align*}
Applying Theorem~\ref{t.ofos}, we know that for $t\gg1$: $\rhob(t,y;0,\cdot)$ reaches equilibrium, and 
\[
\tilde{g}(t,x)\stackrel{\text{law}}{\approx}\frac{e^{\beta \tilde{W}(x)}e^{\beta W(x)}}{\int_{\bT}e^{\beta \tilde{W}(x')}e^{\beta W(x')}dx'},
\]
where $\tilde{W}$ is a standard Brownian bridge independent from $W$.
Combining with \eqref{e.jointgu} (remember $u(0,x)=\beta W'(x)$), we finally conclude that for $t\gg1$,
\begin{equation}\label{e.jointgu1}
\{(g(t,x),u(t,x))\}_{x\in\bT}\stackrel{\text{law}}{\approx } \left\{\bigg( \frac{e^{\beta \tilde{W}(x)}e^{\beta W(x)}}{\int_{\bT}e^{\beta \tilde{W}(x')}e^{\beta W(x')}dx'}, -\beta W'(x)\bigg) \right\}_{x\in\bT}.
\end{equation}

To make the above heuristics rigorous, there are two technical difficulties to deal with: (i) the time-reversal skew symmetry \eqref{e.burgersAsym} only holds at stationarity, while the Burgers drift in \eqref{e.diff1} is not at equilibrium. As discussed already, the Burgers equation mixes exponentially fast on a torus, so one can replace the $u$ in \eqref{e.diff1} by its stationary version with a controllable error. This is done in \cite[Section 3]{YGTK23}. (ii) In the white noise case, the drift $u$ is distribution-valued, the SDEs \eqref{e.diff1}, \eqref{e.diff2} and \eqref{e.diff3} are all formal, and the backward and forward equations \eqref{e.bapde},  \eqref{e.fopde} contain distribution-valued coefficients. We proceeded with an approximation and borrowed the well-developed results for
singular diffusions with distribution-valued drifts. This was done in \cite[Section 4]{YGTK23}.

By \eqref{e.jointgu1}, we have the following (joint) approximation for $s,t-s\gg1$
\[
\begin{aligned}
\rho_{\rm m}(t,-|s,y) = \frac{ \G_{t,s}(-,y) e^{\int_0^y   \tilde{u}(s,w)dw}}{\int_{\bT} \G_{t,s}(-,y')e^{\int_0^{y'}  \tilde{u}(s,w)dw}dy'}&\approx \frac{\rhob(t,\mathrm{m};s,y) e^{\int_0^y   u(s,w)dw}}{\int_{\bT} \rhob(t,\mathrm{m};s,y')e^{\int_0^{y'}  u(s,w)dw}dy'}\\
&\stackrel{\text{law}}{\approx }\frac{e^{\beta \bar{W}(y)}e^{-\beta W(y)}}{\int_{\bT} e^{\beta \bar{W}(y')}e^{-\beta W(y')}dy'},
\end{aligned}
\]
where we used Theorem~\ref{t.ofos} to approximate the $\rhob$ factor, and 
\[
\begin{aligned}
\psi(s,y')-\psi(s,y)=\int_y^{y'}(g(s,z)-1)dz\stackrel{\text{law}}{\approx }\int_y^{y'}\left(\frac{e^{\beta \tilde{W}(z)}e^{\beta W(z)}}{\int_{\bT}e^{\beta \tilde{W}(x')}e^{\beta W(x')}dx'}-1\right) dz.
\end{aligned}
\]
Here $W,\tilde{W},\bar{W}$ are three independent copies of standard Brownian bridges. Using these approximations in \eqref{e.i1minusi2}, one can derive the variance formula \eqref{e.exvarwinding}. For the details, see \cite[Section 5]{YGTK23}.

To summarize, the key insight of the argument is to utilize the
interplay between the Gibbs measure formulation   and the diffusion in
a random environment. On the one hand,   the Gibbs measure formulation allows us to express the quenched density explicitly in terms of the Green's function of SHE, which, in large time, relates to the underlying invariant measure. On the other hand, viewing the problem through the lens of diffusion in a random environment enables us to use time-reversal skew symmetry, transforming the non-Markovian drift into a Markovian one. The diffusion in \eqref{e.diff2} is the continuous analogue of the second class particle in the simple exclusion process, so our approach may be compared to the study of geodesics and competing interfaces in the context of last passage percolation \cite{Sep20}.

\section{Explicit formula for the corrector}
\label{s.corrector}

Recall that in Section~\ref{s.height}, we discussed the previous
studies on the fluctuations of the height function, and presented two
proofs of the central limit theorem \eqref{e.clt}. The first proof
employs a  homogenization argument to analyze the additive functional
of the projective process \eqref{e.additivefunctional}. The key there
is to construct the corrector by solving a Poisson equation,
which allows us to decompose  the (centered) additive functional   as
the sum of co-boundary terms and a martingale. The second proof
utilizes the Clark-Ocone formula,  which  approximates the centered
height function by an explicit martingale \eqref{e.deadditive1}.  Both
proofs rely on the martingale central limit theorem, but the
martingales differ: one is expressed in terms of the corrector, while
the other involves the Malliavin derivative.  Exploring the connection
between these two approaches and the relation between the martingales
is a natural next step. Notably, the corrector $\chi$ solving the
Poisson equation \eqref{e.poisson} and its Frechet gradient
$\mathcal{D}\chi$ appearing in \eqref{e.deadditive} admit explicit expressions: 
\begin{theorem}\label{t.corrector}
Suppose that  $\chi$ solves the
Poisson equation \eqref{e.poisson} and $\mathcal{D}\chi$ is as in
\eqref{e.deadditive}. Then,
\begin{equation}\label{e.exchi}
\chi(\rho)=\frac{2}{\beta^2}\bigg(\E_{\rho_1}\E_{\rho_2}\log \int_{\bT^d}
\rho_1(y)\rho_2(y)  dy- \E_{\rho_1}\log \int_{\bT^d} \rho(y)
\rho_1(y)  dy\bigg),\qquad  \mbox{$\pi_\infty$ a.s}.
\end{equation}
If, in addition, we assume that $\hat{R}(n)>0$ for all $n\in\Z^d$, then
\begin{equation}\label{e.exDchi}
\mathcal{D}\chi(\rho,y)=\frac{2}{\beta^2}-\frac{2}{\beta^2}\E_{\rho_1}
\frac{\rho_1(y)}{\int_{\bT^d} \rho(y')\rho_1(y') dy'}, \qquad  \mbox{$\pi_\infty$ a.s}.
\end{equation}
Here $\rho_1,\rho_2$ are sampled independently from the
invariant measure $\pi_\infty$ of the projective process, and $\E_{\rho_1}, \E_{\rho_2}$ are
the separate expectations on them. 
\end{theorem}

\proof
The fact that both $\chi,\mathcal{D}\chi$ are well-defined follows from the moment bounds provided in \eqref{e.mmbdrho}. The proof relies on a subtle comparison of the two martingale decompositions. One might expect that it is easier to deal with $\mathcal{D}\chi$ first since it appears in the martingale term which constitutes the main contribution to the height fluctuations. In contrast, the corrector $\chi$ itself only appears in co-boundary terms and is much ``smaller'', vanishing in the central limit theorem. Indeed, we will first identify the expression of $\mathcal{D}\chi$.

 From now on, we assume the KPZ equation starts from stationarity, namely, 
\[
Z(0,x)=e^{H(x)},
\]
with $H$ sampled from the invariant measure for the KPZ equation with $H(0)=0$. Since we have different sources of randomnesses involved in the discussion, including the noise $\xi$, the initial data $H$, and the auxiliary $\rho_1,\rho_2$ (in the statement of Theorem~\ref{t.corrector}), throughout  the rest of the section, we use $\EE$ as the expectation on all randomnesses, and $\E$ as the expectation on $\rho(0)=\frac{e^{H}}{\int e^{H}}=:\rho$.

In \eqref{e.deadditive2}, we pick $x=0$ so it becomes \begin{equation}
\begin{aligned}
h(t,0)-\gamma_L(\beta)t=&\log \bar{Z}_0+\log \rho(t,0)-\tfrac12\beta^2\chi(\rho(0))+\tfrac12\beta^2\chi(\rho(t))\\
&+\beta\int_0^t\int_{\bT^d}\rho(s,y)[1-\tfrac12\beta^2\D\chi(\rho(s),y)]\xi(s,y)dyds,
\end{aligned}
\end{equation}
where $\{\rho(t)\}_{t\geq0}$ is the projective process starting at stationarity.

In the stationary setting, we claim that 
\begin{equation}\label{e.meanh}
\EE h(t,0)=\gamma_L(\beta)t.
\end{equation} To see it holds,  note 
\[
\EE h(t,0)=\EE \log Z(t,0)=\EE \log \rho(t,0)+\EE \log \bar{Z}_t.
\]
Since $\{\rho(t,\cdot)\}_{t\geq0}$ starts from stationarity, we have 
\[
\EE \log \rho(t,0)=\EE \log \rho(0,0)=-\EE \log \bar{Z}_0.
\] In addition, by \eqref{e.semilogZ}, we have 
\[
\EE \log \bar{Z}_t-\EE \log \bar{Z}_0=\gamma_L(\beta )t,
\]
which proves the claim.

Define $\mathcal{B}:D^\infty(\bT^d)\to\R$ as
\begin{equation}
\mathcal{B}(f)=\log f(0),
\end{equation}
so 
\[
\log \bar{Z}_0=\log \int_{\bT^d} e^{H(x)}dx=-\mathcal{B}(\rho(0)),
\]
where we used the assumption of $H(0)=0$ so that
\[
\rho(0,0)=\frac{e^{H(0)}}{\int_{\bT^d}e^{H(x)}dx}=\frac{1}{\int_{\bT^d}e^{H(x)}dx}.
\]
Using the operator $\mathcal{B}$ and \eqref{e.meanh}, we can rewrite the decomposition as 
\begin{equation}\label{e.deadditive3}
\begin{aligned}
h(t,0)-\EE h(t,0)=&\mathcal{B}(\rho(t))+\tfrac12\beta^2\chi(\rho(t))-\mathcal{B}(\rho(0))-\tfrac12\beta^2\chi(\rho(0))\\
&+\beta\int_0^t\int_{\bT^d}\rho(s,y)[1-\tfrac12\beta^2\D\chi(\rho(s),y)]\xi(s,y)dyds.
\end{aligned}
\end{equation}

The decomposition in \eqref{e.deadditive3} is what we obtained from the homogenization argument, in the stationary setting: the centered height function is written as the sum of the co-boundary terms (on the first line of the r.h.s.) plus a martingale (the second line).

On the other hand, by the Clark-Ocone formula, the following decomposition holds in the stationary setting  \cite[Proposition 3.1]{ADYGTK22}:\footnote{The result is for the $1+1$ spacetime white noise, but the proof applies to the case of a smooth noise in arbitrary dimensions.}
\begin{equation}\label{e.deadditive4}
\begin{aligned}
h(t,0)-\EE h(t,0)=\mathcal{Y}(\rho(0))-\mathcal{Y}(\rho(t))+\beta\int_0^t\int_{\bT^d}\mathcal{A}(\rho(s),y)\rho(s,y)\xi(s,y)dyds,
\end{aligned}
\end{equation}
with $\mathcal{A}:D^\infty(\bT^d)\times \bT^d\to\R$ and $\mathcal{Y}:D^\infty(\bT^d)\to\R$  defined as 
\begin{equation}\label{e.defAY}
\begin{aligned}
\mathcal{A}(f,y)&=\E_{\rho_1}\frac{\rho_1(y)}{\int_{\bT^d} \rho_1(y')f(y')dy'},\\
\mathcal{Y}(f)&=\E_{\rho_1} \log \int_{\bT^d} \frac{f(y)}{f(0)} \rho_1(y)dy,
\end{aligned}
\end{equation}
where $\rho_1$ is sampled from the invariant measure of the projective process, independent of the noise $\xi$ and the initial data $\rho(0)$, and the expectation $\E_{\rho_1}$ is on $\rho_1$.

Combining \eqref{e.deadditive3} and \eqref{e.deadditive4}, we obtain the identity 
\begin{equation}\label{e.combine}
\begin{aligned}
&\mathcal{B}(\rho(t))+\tfrac12\beta^2\chi(\rho(t))+\mathcal{Y}(\rho(t))-\mathcal{B}(\rho(0))-\tfrac12\beta^2\chi(\rho(0))-\mathcal{Y}(\rho(0))\\
&=\beta\int_0^t\int_{\bT^d}\bigg(\mathcal{A}(\rho(s),y)-[1-\tfrac12\beta^2\D\chi(\rho(s),y)]\bigg)\rho(s,y)\xi(s,y)dyds.
\end{aligned}
\end{equation}
The above identity relates the co-boundary terms to a martingale term, and, by the exponential mixing of the projective process, one might expect that the l.h.s. is of order $O(1)$ while the martingale term  grows as $t\to\infty$ and is   of order $O(\sqrt{t})$. But this is a contradiction, so the martingale term must vanish for all $t>0$, which leads to the explicit expression of $\mathcal{D}\chi$.

More precisely, since $\{\rho(t)\}_{t\geq0}$ is at stationarity, by \eqref{e.mmbdrho} and the bound on the corrector which is a consequence of \eqref{e.Ptsmall}, we derive that, for any $p\geq1$, 
\begin{equation}\label{e.bdlhs}
\EE |\text{ l.h.s. of \eqref{e.combine}}|^p \leq C,
\end{equation}
uniformly in $t\geq0$. On the other side, by It\^o isometry and the fact that $\rho(t)$ is at stationarity,  we obtain the second moment of the r.h.s. of \eqref{e.combine} is given by 
\[
\beta^2 t\,\E \int_{\bT^{2d}}F(\rho,y_1)F(\rho,y_2)R(y_1-y_2)dy_1dy_2,
\]
with 
\[
F(\rho,y):=\bigg(\mathcal{A}(\rho,y)-[1-\tfrac12\beta^2\D\chi(\rho,y)]\bigg)\rho(y).
\]
Therefore, we  must have 
\[
\E \int_{\bT^{2d}}F(\rho,y_1)F(\rho,y_2)R(y_1-y_2)dy_1dy_2=0.
\]
Otherwise, it contradicts with \eqref{e.bdlhs}. By the assumption of $\hat{R}(n)>0$ for all $n\in\Z^d$, we derive that, almost surely, $F(\rho,y)=0$ almost everywhere on $\bT^d$. As a consequence, \[
\D\chi(\rho,y)=\frac{2}{\beta^2}(1-\mathcal{A}(\rho,y))=\frac{2}{\beta^2}-\frac{2}{\beta^2}\E_{\rho_1}\frac{\rho_1(y)}{\int_{\bT^d}\rho_1(y')\rho(y')dy'},
\]
where $\E_{\rho_1}$ is the expectation on $\rho_1$. So \eqref{e.exDchi} is proved.

To show \eqref{e.exchi}, we first observe that \eqref{e.combine} becomes 
\[
\mathcal{B}(\rho(t))+\tfrac12\beta^2\chi(\rho(t))+\mathcal{Y}(\rho(t))-\mathcal{B}(\rho(0))-\tfrac12\beta^2\chi(\rho(0))-\mathcal{Y}(\rho(0))=0.
\]
Recall that $\cP_t$ is the semigroup associated with the projective process, taking expectation with respect to $\xi$ on both side, we obtain
\[
\cP_t\mathcal{B}(\rho)+\tfrac12\beta^2\cP_t \chi(\rho)+\cP_t\mathcal{Y}(\rho)=\mathcal{B}(\rho)+\tfrac12\beta^2\chi(\rho)+\mathcal{Y}(\rho),
\]
which holds for all $t>0$.
Since the process $\rho(t)$ mixes exponentially fast, one expect that as $t\to\infty$, the l.h.s. of the above identity converges to the mean. As a matter of fact, by the same proof for \eqref{e.Ptsmall}, we derive that as $t\to\infty$,
\[
\cP_t\mathcal{B}(\rho)+\tfrac12\beta^2\cP_t \chi(\rho)+\cP_t\mathcal{Y}(\rho)\to \E [\mathcal{B}(\rho)+\mathcal{Y}(\rho)].
\]
(Note that $\E \chi(\rho)=0$ by definition) Thus, we conclude 
\[
\chi(\rho)=\frac{2}{\beta^2}\big(\E [\mathcal{B}(\rho)+\mathcal{Y}(\rho)]-\mathcal{B}(\rho)-\mathcal{Y}(\rho)\big).
\]
One can write $\mathcal{B}+\mathcal{Y}$ as 
\[
\mathcal{B}(\rho)+\mathcal{Y}(\rho)=\E_{\rho_1}\log \int_{\bT^d} \rho(y) \rho_1(y) dy,
\]
which completes the proof of \eqref{e.exchi}.
\qed

\begin{remark}
It is worth noting that the Poisson equation
\begin{equation}
  \label{011608-24}
  \mathcal{L}\chi=-\tilde{\cR}
\end{equation}
is infinite dimensional, hence the
existence of an explicit solution is quite surprising. The explicit expressions of $\chi,\mathcal{D}\chi$ may find applications in other contexts, see e.g. the recent work \cite{BT24}.
\end{remark}

\begin{remark}
Since, we have
already established  in  \eqref{e.exchi}  the formula for the corrector, it is quite
interesting to  verify directly that the Poisson equation indeed holds, using the
expression for $\chi$. In fact the definition of the
functional makes sense for all  $\rho\in D^\infty(\bT^d)$.
Fix any $\rho$, let
$
\rho(t,x;\rho):=Z(t,x;\rho)/\bar Z(t;\rho),
$
where $Z(t,x;\rho)$ solves  \eqref{e.she} with the initial data $Z(0,x;\rho)=\rho(x)$ and $\bar{Z}$ is the spatial integral of $Z$.
Then, by the definition of $\chi$ we have
\begin{equation}\label{e.8231}
\chi\big(\rho(t;\rho)\big)=\frac{2}{\beta^2}  \Big( \log
  \int_{\bT^d}Z(t,x;\rho)   dx\Big)-\frac{2}{\beta^2} \Big(\E_{\rho_1} \log
                 \int_{\bT^d}Z(t,x;\rho) \rho_1(x)  dx\Big)+const.
\end{equation}
Using the It\^o formula, as in \eqref{010608-24}, it
can be concluded that
\begin{align}
  \label{031608-24}
\log
  \int_{\bT^d}Z(t,x;\rho)   dx= \beta
                      \int_0^t\int_{\bT^d}\rho(s,y;\rho)\xi(s,y)dyds -\tfrac{ \beta^2}{2}\int_0^t{\cal R}\big(\rho(s;\rho)\big) ds.
\end{align}
To deal with the second term on the r.h.s. of \eqref{e.8231}, 
 by the properties of the noise itself, for each $\rho$
fixed, we have
$$
\log
                 \int_{\bT^d}Z(t,x;\rho) \rho_1(x)
                 dx\stackrel{{\rm law}}{=}\log
                 \int_{\bT^d}Z(t,x;\rho_1) \rho(x)
                 dx,
$$
with
\begin{align*}
&\log
  \int_{\bT^d}Z(t,x;\rho_1) \rho(x)  dx=\log
  \int_{\bT^d}\rho(t,x;\rho_1) \rho(x)  dx+ \log
  \int_{\bT^d}Z(t,x;\rho_1)  dx.
\end{align*}
Again, by the It\^o formula,
\begin{align*}
\log
  \int_{\bT^d}Z(t,x;\rho_1)   dx= \beta
                      \int_0^t\int_{\bT^d}\rho(s,y;
                 \rho_1)\xi(s,y)dyds -\tfrac{ \beta^2}{2}\int_0^t{\cal R}\big(\rho(s;\rho_1)\big)
    ds,
\end{align*}
and, as a result (due to stationarity of $\big(\rho(t;\rho_1)\big)$),
\begin{align}
  \label{041608-24}
&\E_{\rho_1}\EE \log
                 \int_{\bT^d}Z(t,x;\rho) \rho_1(x)
                 dx=\E_{\rho_1}\EE \log
  \int_{\bT^d}Z(t,x;\rho_1) \rho(x)  dx\notag\\
  &
    =\E_{\rho_1}\log
                 \int_{\bT^d}\rho_1(x) \rho(x)  dx   
  -\tfrac{ \beta^2t}{2}\int_{D^\infty(\bT^d)}{\cal R}(\rho_1)\pi_\infty(d\rho_1).
\end{align}
Therefore, combining \eqref{031608-24} and \eqref{041608-24}, we get 
\begin{align*}
\frac{d}{dt} 
                 \EE\chi\big(\rho(t;\rho)\big) =
  \int_{D^\infty(\bT^d)}{\cal
                 R}(\rho_1)\pi_\infty(d\rho_1) 
    - {\cal P}_t{\cal R}(\rho).
\end{align*}
Substituting $t=0$, \eqref{011608-24} follows.

\end{remark}

\begin{remark}
  \label{rmk5.4}
Finding the Fr\'echet gradient formally from \eqref{e.exchi} leads
to a formula containing only the second term on the right hand side of
\eqref{e.exDchi}. 
A reader might find this fact somewhat puzzling. The explanation is
that the formula for the corrector is given on a ``thin'' subset
(i.e.    closed  with an empty interior) of
   $C(\bT^d)$, that is contained in  $D^\infty(\bT^d)$ - the set of all continuous densities on the
torus $\bT^d$.  Therefore the formula by itself does not allow to
calculate the   Fr\'echet gradient.  To obtain  \eqref{e.exDchi} we need a suitable
extension of the  corrector to an open subset containing 
$D^\infty(\bT^d)$, and one possible choice is to add the term $\frac{2}{\beta^2}\int_{\bT^d}(\rho(y)-1)dy$. 
\end{remark}

\begin{remark}\label{r.8231}
 Using It\^o formula combined with the Clark-Ocone formula, one can verify
that $\chi \big(\rho(t;\rho)\big)$ given by \eqref{e.exchi} satisfies  \eqref{e.deadditive}  
 for any starting point  $\rho\in D^\infty(\bT^d)$ (not necessarily
 starting at stationarity), with $\mathcal{D}\chi$ given by \eqref{e.exDchi}. This fact holds even without the
 assumption on the positivity of the Fourier coefficients of  $R(\cdot)$. We chose to present the proof of Theorem~\ref{t.corrector} under the more restrictive assumption in order to illustrate how the expressions of the corrector and the gradient of corrector were obtained.

 For example we
might proceed as follows.
We define an extension of the corrector \eqref{e.exchi}
on the open set ${\cal C}\subset C(\bT^d)$, consisting of all strictly
positive continuous 
functions $\rho$ on the torus, by letting
\begin{equation}\label{e.exchi1}
  \begin{split}
    &
    \chi(\rho)=\frac{2}{\beta^2}\bigg\{\int_{\bT^d}
\big(\rho(y)-1\big) dy+\hat\E\tilde{\E} \log \int_{\bT^d}
\hat\rho(y)\tilde{\rho}(y)  dy \\
&
-\tilde{\E} \log  \int_{\bT^d} \rho(y) \tilde{\rho}(y)
dy \bigg\} ,\quad \rho\in {\cal C}.
\end{split}
\end{equation}
Recall that the support of $\pi_\infty$ - the invariant measure for
the projective process - is contained in $D_c(\bT^d)$ - the set of all continuous densities on the
torus $\bT^d$.  
A simple estimate shows that
\begin{equation}
 -\infty<\log\Big(\inf_{x\in\bT^d}\rho\Big) \le \int_{D_c(\bT^d) }\log \int_{\bT^d} \rho(y)
\tilde{\varrho}(y)  dy \pi_\infty(d\tilde\rho)\le  \log\|\rho\|_{\infty}<+\infty,
\end{equation}
so $\chi$ is well defined and finite on ${\cal C}$.
It obviously agrees with \eqref{e.exchi} on $D_c(\bT^d)$ and 
its Fr\'echet gradient on the entire set ${\cal C}$  is  given by
the expression \eqref{e.exDchi}. Using It\^o formula and an approximation argument one can verify
that $\chi \big(\rho(t;\rho)\big)$ satisfies  \eqref{e.deadditive}  
 for any starting point  $\rho\in D_c(\bT^d)$ (not necessarily
 starting at stationarity). This fact holds even without the
 assumption on positivity of the Fourier coefficients of the
 covariance function $R(\cdot)$.

  \end{remark}

\section{Law of iterated logarithm}
\label{s.lil}

In this section, we prove the law of iterated logarithm for the height function $h(t,x)$, when it starts from equilibrium, namely $h(0,x)=H(x)$ with $H$ sampled from the invariant measure of the KPZ equation with $H(0)=0$. The problem on the whole line was studied in \cite{DP23}. Here is the main result:
\begin{theorem}\label{t.lil}
With $\rho(t,\cdot)$ started from stationariy, and $\gamma_L(\beta),\sigma_L(\beta)$ given in \eqref{e.congamma} and \eqref{e.exvar}, we have
\begin{equation}\label{e.lil}
\begin{aligned}
&\limsup_{t\to\infty} \frac{h(t,x)-\gamma_L(\beta)t}{\sqrt{2t\log \log t}}=\sigma_L(\beta),\\
&\liminf_{t\to\infty} \frac{h(t,x)-\gamma_L(\beta)t}{\sqrt{2t\log \log t}}=-\sigma_L(\beta),
\end{aligned}
\end{equation}
almost surely.
\end{theorem}

The proof relies on the martingale representation \eqref{e.deadditive2}:
\[
h(t,x)-\gamma_L(\beta)t=\mathscr{C}_t+ \mathscr{M}_t,
\]
with 
\begin{equation}\label{e.defCM}
\begin{aligned}
&\mathscr{C}_t(x)=\log \bar{Z}_0+\log \rho(t,x)-\tfrac12\beta^2\chi(\rho(0))+\tfrac12\beta^2\chi(\rho(t)),\\
&\mathscr{M}_t=\beta\int_0^t\int_{\bT^d}\rho(s,y)[1-\tfrac12\beta^2\D\chi(\rho(s),y)]\xi(s,y)dyds.
\end{aligned}
\end{equation}
To show Theorem~\ref{t.lil}, it is enough to prove that \eqref{e.lil} holds with $h(t,x)-\gamma_L(\beta)t$ replaced by $\mathscr{M}_t$, and  the following lemma holds:
\begin{lemma}\label{l.Cgozero}
\begin{equation}\label{e.Cgozero}
\lim_{t\to+\infty}\frac{\mathscr{C}_t(x)}{\sqrt{2t\log \log t}}=0
\end{equation}
almost surely.
 \end{lemma}
 
 To prove the above lemma, we need an improvement of \eqref{e.mmbdrho}:
 \begin{lemma}
 For any $p\geq1$, we have  
 \begin{equation}\label{e.mmbdnew}
   C_p:=\sup_{t\geq 1}\EE\Big[ \sup_{s\in[t,t+1]} \sup_{x\in\bT^d} \{\rho(s,x)^p+\rho(s,x)^{-p}\}\Big]<+\infty.
 \end{equation}
 \end{lemma}
 
 \begin{proof}
 Recall that $\rho(s,x)=\frac{Z(s,x)}{\int_{\bT^d} Z(s,x')dx'}$. Using the Green's function of SHE, for any $s\geq1$, we can rewrite it as 
 \[
 \rho(s,x)=\frac{\int_{\bT^d} \G_{s,s-1}(x,y)Z(s-1,y)dy}{\int_{\bT^{2d}} \G_{s,s-1}(x',y)Z(s-1,y)dx'dy},
 \]
 which leads to the natural bound 
 \[
 \begin{aligned}
\inf_{x,y} \G_{s,s-1}(x,y)&(\sup_{x,y}\G_{s,s-1}(x,y))^{-1} \\
&\leq  \rho(s,x) \leq \sup_{x,y} \G_{s,s-1}(x,y)(\inf_{x,y} \G_{s,s-1}(x,y))^{-1}.
 \end{aligned}
 \]
 Thus, to show \eqref{e.mmbdnew}, it suffices to show the following moment bounds on the Green's function of SHE: for any $p\geq1$, 
 \[
 \EE \Big[ \sup_{s\in[1,2],x,y\in\bT^d}\G_{s,s-1}(x,y)^p\Big]+\EE \Big[ \sup_{s\in[1,2],x,y\in\bT^d}\G_{s,s-1}(x,y)^{-p}\Big] \leq C_p.
 \]
 This follows from a chaining argument, see e.g. \cite[Proof of Lemma 4.1]{GK21}. We do not provide the details here.
 \end{proof}
 
\begin{proof}[Proof of Lemma~\ref{l.Cgozero}] 
By \eqref{e.mmbdnew},   for any $n\in\Z_+$ and $p\geq 1$, we have
 \[
 \EE\Big[ \sup_{t\in[n,n+1]}|\log \rho(t,x)|^p\Big] \leq C_p.
 \]
 Additionally, by \cite[Proposition 5.5]{GK21}, we have 
 \[
 |\chi(\rho(t))| \leq C(1+\sup_x \rho(t,x)).
 \]
Using \eqref{e.mmbdnew} again, we obtain
 \[
 \EE\Big[  \sup_{t\in[n,n+1]}|\chi(\rho(t))|^p \Big]\leq C_p.
 \]
To apply the Borel-Cantelli lemma, we consider any sequence $c_n\to0$ as $n\to\infty$, so 
\begin{equation}\label{e.boca}
\begin{aligned}
&\sum_{n\geq0} \PP\big[\sup_{t\in[n,n+1]}|\log \rho(t,x)|+\sup_{t\in[n,n+1]}|\chi(\rho(t))|>c_n\sqrt{n}\big]\\
&\leq \sum_{n\geq0} \frac{1}{(c_n\sqrt{n})^p} \EE \big[(\sup_{t\in[n,n+1]}|\log \rho(t,x)|+\sup_{t\in[n,n+1]}|\chi(\rho(t))|)^p\big] \leq \sum_{n\geq0} \frac{C_p}{(c_n\sqrt{n})^p}.
\end{aligned}
\end{equation}
Choose $c_n=n^{-\alpha}$ with $\alpha<1/2$ and $p$ such that $p(\frac12-\alpha)>1$, then the above summation is finite, from which we conclude that $\frac{1}{\sqrt{t}}(|\log \rho(t,x)|+ |\chi(\rho(t))|)\to0$ almost surely, which completes the proof.
\end{proof}
 

To complete the proof of Theorem~\ref{t.lil}, it remains to show the law of iterated logarithm for the martingale $\mathscr{M}_t$:
\begin{equation}\label{e.lilma}
\begin{aligned}
&\limsup_{t\to\infty} \frac{\mathscr{M}_t}{\sqrt{2t\log \log t}}=\sigma_L(\beta),\\
&\liminf_{t\to\infty} \frac{\mathscr{M}_t}{\sqrt{2t\log \log t}}=-\sigma_L(\beta).
\end{aligned}
\end{equation}
First, note that the bracket process is \[
\la \mathscr{M}\ra_t= \int_0^t  \mathcal{T}(\rho(s))ds,
\]
with 
\begin{equation}\label{e.defT}
  \mathcal{T}(
  \rho):=\beta^2\int_{\bT^{2d}}\prod_{j=1}^2
  f(y_j)[1-\tfrac12\beta^2\D\chi(\rho,y_j)] R(y_1-y_2)dy_1dy_2.
\end{equation} 
To see $\mathcal{T}(\rho)\in L^1(\Omega)$, we use  \eqref{e.exDchi}, from
which one can actually derive that $\mathcal{D}\chi(\rho,y)$ is
bounded almost surely. Indeed, by Jensen's inequality and \eqref{e.mmbdrho}, we have 
\[
\begin{aligned}
\E_{\rho_1} \frac{\rho_1(y)}{\int_{\bT^d} \rho(y')\rho_1(y') dy'}\leq\E_{\rho_1} \rho_1(y)\sup_{y'}  \rho_1(y')^{-1} \leq C.
\end{aligned}
\]
Thus, $\mathcal{T}(\rho)$ is bounded almost surely  by 
\[
|\mathcal{T}(\rho)| \leq C \int_{\bT^{2d}}\rho(y_1)\rho(y_2) R(y_1-y_2)dy_1dy_2.
\]
Applying \eqref{e.mmbdrho}, we obtain $\EE \mathcal{T}(\rho)\leq
C$. Applying the Birkhoff’s Ergodic Theorem we conclude that
\begin{equation}\label{e.conqv}
\frac{1}{t}\la \mathscr{M}\ra_t\to \sigma_L^2(\beta)
\end{equation}
almost surely. Since $\mathscr{M}_t$ is a continuous martingale and $\la\mathscr{M}\ra_\infty=\infty$ almost surely, by the Dambis-Dubins-Schwarz theorem \cite[Theorem 5.13]{legall}, there exists a standard Brownian motion $\beta$ such that almost surely, 
\[
\mathscr{M}_t=\beta_{\la \mathscr{M}\ra_t}, \quad\quad t\geq0.
\]
By the law of iterated logarithm for Brownian motion, the following almost sure convergence holds: 
\[
\begin{aligned}
&\limsup_{t\to\infty}\frac{\beta_{\la \mathscr{M}\ra_t}}{\sqrt{2\la \mathscr{M}\ra_t\log \log \la \mathscr{M}\ra_t}}=1,\\
&\liminf_{t\to\infty}\frac{\beta_{\la \mathscr{M}\ra_t}}{\sqrt{2\la \mathscr{M}\ra_t\log \log \la \mathscr{M}\ra_t}}=-1.
\end{aligned}
\]
Together with \eqref{e.conqv}, we complete the proof of \eqref{e.lilma}.

\begin{remark}
To prove Theorem~\ref{t.lil}, applying the Birkhoff's Ergodic Theorem in \eqref{e.conqv} is the only place where we used the assumption of $\{\rho(t)\}_{t\geq0}$ starting at stationarity.  For the general case where the projective process does not necessarily start from stationarity, we sketch a possible argument below. Consider another projective process $\{\tilde{\rho}(t)\}_{t\geq0}$, which starts at stationarity and is driven by the same noise. By the previous discussion, we already have 
\[
\frac{1}{t}  \int_0^t \mathcal{T}(\tilde{\rho}(s))ds\to \sigma_L^2(\beta)
\]
almost surely, so it remains to show that 
\[
\frac{1}{t} \int_0^t [\mathcal{T}(\rho(s))-\mathcal{T}(\tilde{\rho}(s))] ds\to0
\]
almost surely. By the   application of the Borel-Cantelli lemma as in \eqref{e.boca}, it reduces to proving 
\begin{equation}\label{e.nonbirk}
 \int_0^\infty \EE |\mathcal{T}(\rho(s))-\mathcal{T}(\tilde{\rho}(s))|^p ds \leq C_p.
\end{equation}
 Applying \eqref{051705-23}, we have for any $p\geq1$,
\[
\EE \sup_{x\in\bT^d} |\rho(s,x)-\tilde{\rho}(s,x)|^p \leq C e^{-\lambda s},
\]
thus, from the expression of $\mathcal{T}$ in \eqref{e.defT}, the only
ingredient we are missing to control the error in \eqref{e.nonbirk} is
an estimate of $\mathcal{D}\chi(v_1,y)-\mathcal{D}\chi(v_2,y)$ for
$v_1,v_2\in D^\infty(\bT^d)$. This could be achieved for example by combining
\cite[Lemma 5.9, Corollary 5.11]{GK21}, or using the explicit formula for
$\mathcal{D}\chi$ furnished in \eqref{e.exDchi} (recall that it is
also valid out of stationarity by Remark~\ref{r.8231}). 
\end{remark}

\section{Open problems}
\label{s.open}

In this section, we present several open problems   worth further exploration.

\begin{enumerate}
\item \emph{Improve Theorem~\ref{t.123}}.  The goal is to extend the
  theorem to cover both the relaxation and sub-relaxation
  regimes. Using the notations there, with $L\sim t^\alpha$, the
  regime of $\alpha=2/3$ is the relaxation regime, and $\alpha>2/3$ is
  the sub-relaxation regime. We expect that for all $\alpha\geq2/3$,
  the growth of the surface will satisfy  $\mathrm{Var}\, h(t,0)\asymp
  t^{2/3}$. Our current perturbative approach does not   cover these regimes.

\item \emph{Open KPZ equation}. We expect an analogous result of Theorem~\ref{t.123} for the open KPZ equation. The invariant measures in this case are explicitly known \cite{BCY23}, thus the primary focus should be on the $\sigma_L^2(\beta)$ given by \eqref{e.exvar}. Unlike the periodic setting, we deal with an interval $[a,b]$ instead of a torus. Consequently, the asymptotic behavior of $\mathrm{Var} h(t,0)$  depends on whether  $a\to-\infty$, $b\to\infty$, or both, as well as  the boundary conditions   at $a$ and $b$. It is  known that for the half-line open KPZ equation, Gaussian or KPZ fluctuations may occur depending on the boundary conditions \cite{BW23,IMS22}. Thus, investigating the transition based on these conditions would be particularly interesting.

\item \emph{Asymptotics of the effective diffusivity.} For the winding number of the directed polymer in a white noise environment, we have derived a formula for the effective diffusivity $\Sigma(\beta)$ with period $L=1$ (see \eqref{e.exvarwinding}). In general, denote the diffusivity by $\Sigma_L(\beta)$. By scaling, $\Sigma_1(\beta)=\Sigma_L(\beta/\sqrt{L})$. To go from  diffusive to super-diffusive behaviors as $L\to\infty$, it is conjectured that $\Sigma_L(\beta)\sim \sqrt{L}$,  equivalent with $\Sigma_1(\beta)\sim \beta$ as $\beta\to\infty$. An open problem is to prove this conjecture using the formula \eqref{e.exvarwinding} and to further identity the size of the displacement with $L\sim t^{\alpha}$ approaches infinity.

\item \emph{Point-to-point directed polymer}. The proof of the central limit theorems for the height function and the winding number  relies crucially on the exponential mixing of the polymer endpoint process, established in Theorem~\ref{t.ofos}. Besides the spatial periodicity of the random environment, we also assumed that one of the polymer's boundary conditions is periodic. This allows us to work with the polymer path on a cylinder. As a result, our approach does not seem to cover the case of a point-to-point directed polymer. An open problem is to study this scenario, for instance, to prove the central limit theorem for $\log \sfZ(t,0)$ with $\sfZ$ defined in \eqref{e.defsfZ}, and to show that the displacement of the midpoint of the point-to-point directed polymer satisfies a central limit theorem. 

\item \emph{Asymptotics of heat kernel.}    The central limit theorem established in Theorem~\ref{t.winding} only concerns the spatial averaging of $\sfp$ on diffusive scales and does not explore the local behaviors. We expect that a more detailed understanding of $\sfp$ going beyond the CLT could illuminate the transition from diffusive to super-diffusive  behaviors, as well as the shift from homogenization to localization phenomena. For instance, it would be valuable to derive what the analogue of \cite[Equation (28), page 136]{QR14} might be in a periodic setting.
\end{enumerate}


\begin{thebibliography}{99}


  

\bibitem{ACQ11}
{\sc G.~Amir, I.~Corwin, and J.~Quastel}, {\em Probability distribution of the
  free energy of the continuum directed random polymer in {$1+1$} dimensions},
  Comm. Pure Appl. Math., 64 (2011), 466--537.
  
  
  \bibitem{ABK24}
{\sc S.~Armstrong, A.~Bou-Rabee and T.~Kuusi}, {\em Superdiffusive central limit theorem for a Brownian particle in a critically-correlated incompressible random drift}, arXiv preprint arXiv:2404.01115 (2024).
  
  
  \bibitem{bakhtin1}
  {\sc Y.~Bakhtin}, {\em The Burgers equation with Poisson random forcing}, Annals of Probability (2013): 2961-2989.
  
  \bibitem{bakhtin2020localization}
{\sc Y.~Bakhtin and D.~Seo}, {\em Localization of directed polymers in
  continuous space}, Electronic Journal of Probability, 25 (2020), pp.~1--56.


\bibitem{baik21}
{\sc J.~Baik}, {\em KPZ limit theorems}, (2021).

\bibitem{BL16}
{\sc J.~Baik and Z.~Liu}, {\em T{ASEP} on a ring in sub-relaxation time scale},
  J. Stat. Phys., 165 (2016), 1051--1085.

\bibitem{BL18}
{\sc J.~Baik and Z.~Liu}, {\em Fluctuations of {TASEP} on a ring in relaxation
  time scale}, Comm. Pure Appl. Math., 71 (2018), 747--813.

\bibitem{BL19}
{\sc J.~Baik and Z.~Liu}, {\em Multipoint distribution of periodic {TASEP}}, J.
  Amer. Math. Soc., 32 (2019), 609--674.

\bibitem{BL21}
{\sc J.~Baik and Z.~Liu}, {\em Periodic {TASEP} with general initial
  conditions}, Probab. Theory Related Fields, 179 (2021), 1047--1144.

\bibitem{BLS20}
{\sc J.~Baik, Z.~Liu and G.~LF~Silva}, {\em Limiting one-point distribution of periodic TASEP}, Annales de l'Institut Henri Poincare (B) Probabilites et statistiques. Vol. 58. No. 1. Institut Henri Poincaré, 2022.

%
%
%
%

\bibitem{BCY23}
{\sc  G.~Barraquand, I.~Corwin and Z.~Yang},  {\em Stationary measures for integrable polymers on a strip}, arXiv preprint arXiv:2306.05983 (2023).

\bibitem{BD21}
{\sc G.~Barraquand and P.~L. Doussal}, {\em Steady state of the KPZ equation
  on an interval and Liouville quantum mechanics}, Europhysics Letters 137.6 (2022): 61003.
%

\bibitem{BW23}
{\sc G.~Barraquand  and S.~Wang}, {\em An identity in distribution between full-space and half-space log-gamma polymers}, International Mathematics Research Notices 2023.14 (2023): 11877-11929.

\bibitem{bates2020endpoint}
{\sc E.~Bates and S.~Chatterjee}, {\em The endpoint distribution of directed
  polymers}, The Annals of Probability, 48 (2020), pp.~817--871.


\bibitem{berger2017high}
{\sc Q.~Berger and H.~Lacoin}, {\em The high-temperature behavior for the
  directed polymer in dimension $1+ 2$}, in Annales de l'Institut Henri
  Poincar{\'e}, Probabilit{\'e}s et Statistiques, vol.~53, Institut Henri
  Poincar{\'e}, 2017, pp.~430--450.
  
  \bibitem{BC95}
  {\sc L.~Bertini and N. Cancrini}, {\em The Stochastic Heat Equation: Feynman-Kac Formula and Intermittence}, 
  J. Stat. Phys., 78 (1995), 1377--1401.

\bibitem{BG97}
{\sc L.~Bertini and G.~Giacomin}, {\em Stochastic {B}urgers and {KPZ} equations
  from particle systems}, Comm. Math. Phys., 183 (1997), 571--607.
  
  \bibitem{bil} {\sc P.~Billingsley}, {\em Convergence of Probability
    Measures, 2nd edt}, Wiley and Sons, 1999.
  
  \bibitem{BT24}
{\sc A.~Blessing and T.~Rosati}, {\em Quantitative instability for
  stochastic scalar reaction-diffusion equations}, arXiv preprint
arXiv:2406.04651 (2024).




  
  \bibitem{BJ12}
{\sc P.~Bougerol and J.~Lacroix},  {\em Products of random matrices with applications to Schr\"odinger operators}, Vol. 8. Springer Science \& Business Media, 2012.
%
%
%
%




\bibitem{broker2019localization}
{\sc Y.~Br{\"o}ker and C.~Mukherjee}, {\em Localization of the Gaussian
  multiplicative chaos in the Wiener space and the stochastic heat equation in
  strong disorder}, The Annals of Applied Probability, 29 (2019),
  pp.~3745--3785.




  

  

  \bibitem{GK211}
{\sc {\'E}.~Brunet, Y.~Gu and T.~Komorowski}, {\em {High temperature behaviors of the directed polymer on a cylinder}}, arxiv preprint arXiv:2110.07368v3.



\bibitem{BKWW21}
{\sc W.~Bryc, A.~Kuznetsov, Y.~Wang, and J.~Wesolowski}, {\em {Markov processes
  related to the stationary measure for the open KPZ equation}}, Probability Theory and Related Fields 185.1 (2023): 353-389.
  
  
\bibitem{CC18}  
 {\sc G.~Cannizzaro and K.~Chouk}, {\em Multidimensional SDEs with
   singular drift and universal construction of the polymer measure
   with white noise potential}, The Annals of Probability 46.3 (2018):
 1710-1763.

 
 
 \bibitem{CLF22}
{\sc G.~Cannizzaro, L.~Haunschmid-Sibitz and F.~Toninelli}, {\em $\sqrt{\log t}$-Superdiffusivity for a Brownian particle in the curl of the 2D GFF}, The Annals of Probability 50.6 (2022): 2475-2498.

\bibitem{CH06}
{\sc P.~Carmona and Y.~Hu}, {\em Strong disorder implies strong localization for directed polymers in a random environment}, arXiv preprint math/0601670 (2006).
  




\bibitem{Cha14}
{\sc S.~Chatterjee}, {\em Superconcentration and related topics}, Springer
  Monogr. Math., Springer, Cham, 2014.


\bibitem{CMOW22} {\sc  G. Chatzigeorgiou, P. Morfe, F. Otto, L.
    Wang}, {\em   The gaussian free-field as a stream function:
asymptotics of effective diffusivity in infra-red
cut-off}, preprint available at {\tt 
https://doi.org/10.48550/arXiv.2212.14244}

  
  

\bibitem{CKNP19}
{\sc L.~Chen, D.~Khoshnevisan, D.~Nualart, and F.~Pu}, {\em Spatial
  ergodicity for SPDEs via Poincar\'e-type inequalities},
Electronic Journal of Probability 26 (2021): 1-37.
  
  
  \bibitem{CKNP20}
{\sc L.~Chen, D.~Khoshnevisan, D.~Nualart, and F.~Pu}, {\em {Central limit
  theorems for spatial averages of the stochastic heat equation via
  Malliavin-Stein's method}}, Stochastics and Partial Differential Equations: Analysis and Computations (2020): 1-55.
  
  \bibitem{CN95}
{\sc   F.~Comets and J.~Neveu},  {\em The Sherrington-Kirkpatrick model of spin glasses and stochastic calculus: the high temperature case}, Communications in Mathematical Physics 166 (1995): 549-564.


%


%
  
  
%
%

\bibitem{Cor12}
{\sc I.~Corwin}, {\em The {K}ardar-{P}arisi-{Z}hang equation and universality
  class}, Random Matrices Theory Appl., 1 (2012), 1130001, 76.



\bibitem{CH16}
{\sc I.~Corwin and A.~Hammond}, {\em KPZ line ensemble}, Probability Theory and Related Fields 166.1 (2016): 67-185.

\bibitem{CK21}
{\sc I.~Corwin and A.~Knizel}, {\em {Stationary measure for the open KPZ
  equation}}, Communications on Pure and Applied Mathematics 77.4 (2024): 2183-2267.


\bibitem{daza}  {\sc G. Da Prato, J. Zabczyk}, {\em Stochastic equations in infinite dimensions.} Second edition. Encyclopedia of Mathematics and its Applications, 152. Cambridge University Press, Cambridge, 2014. 


\bibitem{DP23}
{\sc S.~Das  and P.~Ghosal}, {\em Law of iterated logarithms and fractal properties of the KPZ equation}, The Annals of Probability 51.3 (2023): 930-986.


\bibitem{das}  
{\sc S.~Das and W.~Zhu}, {\em Localization of the continuum directed random polymer}, arXiv preprint arXiv:2203.03607 (2022). 




\bibitem{DD16}
{\sc F.~Delarue and R.~Diel},  {\em Rough paths and 1d SDE with a time dependent distributional drift: application to polymers}, Probability Theory and Related Fields 165.1-2 (2016): 1-63.


\bibitem{DEM93}
{\sc B.~Derrida, M.~R. Evans and D.~Mukamel}, {\em Exact diffusion constant
  for one-dimensional asymmetric exclusion models}, J. Phys. A, 26 (1993),
  4911--4918.



  

  
%
  
  
  
    \bibitem{ADYGTK22} {\sc A.~Dunlap, Y.~Gu and T.~Komorowski}, {\em Fluctuation exponents of the KPZ equation on a large torus}, Communications on Pure and Applied Mathematics 76, no. 11 (2023): 3104-3149.






\bibitem{FQ15}
{\sc T.~Funaki and J.~Quastel}, {\em K{PZ} equation, its renormalization and
  invariant measures}, Stoch. Partial Differ. Equ. Anal. Comp., 3 (2015),
  159--220.
  
  \bibitem{GH21}
{\sc  Y.~Gu and C.~Henderson}, {\em A PDE hierarchy for directed polymers in random environments},  Nonlinearity 34.10 (2021): 7335.

\bibitem{GH22}
{\sc  Y.~Gu and C.~Henderson}, {\em Long-time behaviour for a nonlocal model from directed polymers}, Nonlinearity 36.2 (2022): 902.
  
\bibitem{GK21}
{\sc Y.~Gu and T.~Komorowski}, {\em {KPZ on torus: Gaussian fluctuations}},
   Ann. Inst. H. Poincare,
     Prob. and Stat., 2024, Vol. 60, No. 3, 1570–1618.

    \bibitem{YGTK22}  {\sc Y.~Gu and T.~Komorowski}, {\em     Fluctuations of the winding
      number of a directed polymer on a cylinder},  
   SIAM Journ. Math. Anal., 55 (2023), No. 4, pp. 3262-3286.

\bibitem{YGTK23}
  {\sc Y.~Gu and T.~Komorowski},  {\em Effective diffusivities in periodic KPZ}, Probability Theory and Related Fields (2024): 1-55.
  
  \bibitem{wrong}
{\sc Y.~Gu and T.~Komorowski}, {\em High temperature behaviors of the directed polymer on a cylinder}, Journal of Statistical Physics 186.3 (2022): 48.
  


\bibitem{GP17}
{\sc M.~Gubinelli and N.~Perkowski}, {\em K{PZ} reloaded}, Comm. Math. Phys.,
  349 (2017), 165--269.
  
  \bibitem{GP20}
 {\sc M.~Gubinelli and N.~Perkowski}, {\em The infinitesimal generator of the stochastic Burgers equation}, Probability Theory and Related Fields 178.3 (2020): 1067-1124.
  
  
  
  \bibitem{MH13}
{\sc M.~Hairer}, {\em Solving the KPZ equation}, Annals of mathematics (2013): 559-664.
  

 \bibitem{HKP-G16}
    {\sc M.~Hairer, L. Koralov, Z. Pajor-Gyulai},
     {\em From averaging to homogenization in cellular flows---an exact
              description of the transition},
   Annales de l'Institut Henri Poincar\'e{} Probabilit\'es et
              Statistiques,
     {52},
       {2016},
     {1592--1613}.


\bibitem{HM18}
{\sc M.~Hairer and J.~Mattingly}, {\em The strong {F}eller property for
  singular stochastic {PDE}s}, Ann. Inst. Henri Poincar{\'e} Probab. Stat., 54
(2018), 1314--1340.



\bibitem{HIKNP-G18}
   {\sc M.~Hairer, G. Iyer, L. Koralov, A. Novikov, Z. Pajor-Gyulai},  
      {\em A fractional kinetic process describing the intermediate time
              behaviour of cellular flows},
   {The Annals of Probability},
    {46},
   {2018},
      {897--955}.

  
  \bibitem{HR23}
{\sc M.~Hairer and T.~Rosati}, {\em Spectral gap for projective processes of linear SPDEs}, arXiv preprint arXiv:2307.07472 (2023).

\bibitem{HZ95}
{\sc T.~Halpin-Healy and Yi-Cheng Zhang}, {\em Kinetic roughening phenomena, stochastic growth, directed polymers and all that}, Aspects of multidisciplinary statistical mechanics. Physics reports 254.4-6 (1995): 215-414.
  
  
  

\bibitem{IMS22}
{\sc T.~Imamura, Takashi, M.~Mucciconi and T.~Sasamoto}, {\em Solvable
  models in the KPZ class: approach through periodic and free boundary
  Schur measures}, arXiv preprint arXiv:2204.08420 (2022).

\bibitem{IKNR14}
 {\sc G. Iyer, T. Komorowski, A. Novikov and L. Ryzhik}, {\em From
    homogenization to averaging in cellular flows}, Annales de
  l'Institut Henri Poincare / Analyse non lineaire. 31 (2014)
  957-983.


\bibitem{JS03}  {\sc J. Jacod and A. N. Shiryaev},
     {\em Limit theorems for stochastic processes},
     {Grundlehren der mathematischen Wissenschaften [Fundamental
              Principles of Mathematical Sciences]},
    {288},  {Second ed.},
 {Springer-Verlag, Berlin},
        {2003},






\bibitem{kardar1986dynamic}
{\sc M.~Kardar, G.~Parisi, and Y.-C. Zhang}, {\em Dynamic scaling of growing
  interfaces}, Physical Review Letters, 56 (1986), p.~889.




\bibitem{DK14}
{\sc D.~Khoshnevisan}, {\em Analysis of stochastic partial differential equations}, Vol. 119. American Mathematical Soc., 2014.

\bibitem{KKM23}
{\sc D.~Khoshnevisan, K.~Kim and C.~Mueller}, {\em Dissipation in parabolic SPDEs II: Oscillation and decay of the solution}, Annales de l'Institut Henri Poincare (B) Probabilites et statistiques. Vol. 59. No. 3. Institut Henri Poincaré, 2023.

\bibitem{KV86}
{\sc C.~Kipnis and SR Srinivasa Varadhan}, {\em Central limit theorem for additive functionals of reversible Markov processes and applications to simple exclusions}, Communications in Mathematical Physics 104.1 (1986): 1-19.


\bibitem{KM22}
{\sc A.~Knizel and K.~Matetski}, {\em The strong Feller property of the open KPZ equation}, arXiv preprint arXiv:2211.04466 (2022).

 \bibitem{KLO12}  {\sc T.~Komorowski, C.~Landim and S.~Olla}, {\em
Fluctuations in Markov Processes. Time Symmetry and Martingale Approximation},
Springer Ser.: {\em Grundlehren der mathematischen Wissenschaften}, Vol. {\bf 345}, 2012.


 \bibitem{KO02} 
 {\sc T.~Komorowski and S.~Olla}, {\em
 On
the Superdiffusive Behavior of Passive Tracer  with a Gaussian
Drift},  Journ. Stat. Phys. {\bf 108}, 647-668, (2002).




\bibitem{krug}
{\sc J.~Krug, and P.~Meakin}, {\em Universal finite-size effects in the rate of growth processes}, Journal of Physics A: Mathematical and General 23.18 (1990): L987.


\bibitem{KS91}
{\sc J.~Krug and H.~Spohn},  {\em Kinetic roughening of growing surfaces}, Solids far from equilibrium, 1.


\bibitem{lacoin2010new}
{\sc H.~Lacoin}, {\em New bounds for the free energy of directed polymers in
  dimension 1+ 1 and 1+ 2}, Communications in Mathematical Physics, 294 (2010),
  pp.~471--503.



\bibitem{legall}
{\sc J.-F.~Le Gall}, {\em Brownian motion, martingales, and stochastic calculus}, Springer International Publishing Switzerland, 2016.






  


\bibitem{MY05}
{\sc H.~Matsumoto and M.~Yor}, {\em Exponential functionals of {B}rownian
  motion. {I}. {P}robability laws at fixed time}, Probab. Surv., 2 (2005),
  312--347.

\bibitem{MP92}
{\sc M.~M\'ezard and G.~Parisi}, {\em A variational approach to directed polymers}, Journal of Physics A: Mathematical and General 25.17 (1992): 4521.


  
  \bibitem{nakashima2014remark}
{\sc M.~Nakashima}, {\em A remark on the bound for the free energy of directed
  polymers in random environment in 1+ 2 dimension}, Journal of Mathematical
  Physics, 55 (2014), p.~093304.

\bibitem{nakashima2019free}
{\sc M.~Nakashima}, {\em Free energy of
  directed polymers in random environment in $1+ 1$-dimension at high
  temperature}, Electronic Journal of Probability, 24 (2019), pp.~1--43.




%
%
\bibitem{Pa22}
{\sc S.~Parekh}, {\em Ergodicity results for the open KPZ equation}, arXiv preprint arXiv:2212.08248 (2022).



\bibitem{PY}
{\sc N.~Perkowski and H.~Yang}, {\em private communication.}

\bibitem{Pr24}
{\sc S.~Prolhac}, {\em KPZ fluctuations in finite volume}, SciPost Physics Lecture Notes (2024): 081.

\bibitem{Qua12}
{\sc J.~Quastel}, {\em Introduction to {KPZ}}, in Current Developments in
  Mathematics, 2011, Int. Press, Somerville, MA, 2012, 125--194.
  
  \bibitem{QR14}
{\sc J.~Quastel and D.~Remenik}, {\em Airy processes and variational problems}, Topics in percolative and disordered systems. New York, NY: Springer New York, 2014. 121-171.


\bibitem{QS15}
{\sc J.~Quastel and H.~Spohn}, {\em The one-dimensional {KPZ} equation and its
  universality class}, J. Stat. Phys., 160 (2015), 965--984.
  
  \bibitem{RT05}
 {\sc C.~Rovira and S.~Tindel},  {\em On the Brownian-directed polymer in a Gaussian random environment}, Journal of Functional Analysis 222.1 (2005): 178-201.
  

\bibitem{Ros21}
{\sc T.~C. Rosati}, {\em {Synchronization for KPZ}}, Stochastics and Dynamics, Vol. 22, No. 04, 2250010 (2022).

  
  
    \bibitem{Sep20}
 {\sc T.~Sepp\"al\"ainen}, {\em Existence, uniqueness and coalescence of directed planar
geodesics: Proof via the increment-stationary growth process}, Ann. Inst. H.
Poincare Probab. Statist. 56(3): 1775-1791, 2020.
  

\bibitem{Sin91}
{\sc Y.~G. Sina{\u\i}}, {\em Two results concerning asymptotic behavior of
  solutions of the {B}urgers equation with force}, J. Stat. Phys., 64 (1991),
  1--12.


  \bibitem{walsh}
  {\sc
  J.~Walsh}, {\em An introduction to stochastic partial differential equations}, \'Ecole d'\'Et\'e de Probabilit\'es de Saint Flour XIV-1984. Springer, Berlin, Heidelberg (1986), pp.~265--439.


\bibitem{Yor92}
{\sc M.~Yor}, {\em On some exponential functionals of {B}rownian motion}, Adv.
  Appl. Probab., 24 (1992), 509--531.





\end{thebibliography}
 \end{document}